\documentclass[a4paper,11pt]{amsart}
\usepackage{amsmath,amssymb,amscd,amsfonts}
\usepackage{amsrefs, esint}
\numberwithin{equation}{section}
\newtheorem{theorem}{Theorem}[section]

\newtheorem{rem}[theorem]{Remark}
\newtheorem{defn}[theorem]{Definition}
\newtheorem{lemma}[theorem]{Lemma}

\newtheorem{prop}[theorem]{Proposition}
\newtheorem{conj}[theorem]{Conjecture}
\newtheorem{cor}[theorem]{Corollary}
\newtheorem{example}[theorem]{Example}

\def\Re{\mathop{\rm Re}\nolimits} 
\def\Im{\mathop{\rm Im}\nolimits}

\def\dbar{\bar\partial}

\def\d{\partial}

\def\cI{{\mathcal I}}

\def\cE{{\mathcal E}}

\def\cK{{\mathcal K}}

\def\cF{{\mathcal F}}

\let\ep=\varepsilon
\let\vp=\varphi 

\def\bC{{\mathbb C}}
\def\bR{{\mathbb R}}

\def\bD{{\mathbb D}}

\def\bP{{\mathbb P}}
\def\a{{\alpha}}
\def\z{{\zeta}}

\def\b{\beta}

\def\k{{\kappa}}

\def\l{{\ell}}

\def\t{{\tau}}

\title[Zero mass conjecture]
{ On the residual Monge-Amp\`{e}re mass of plurisubharmonic functions with symmetry in $\bC^2$ }

\author{Long Li}
\address{ IMS of ShanghaiTech University, 393 Middle Huaxia Road, Pudong 201210, Shanghai, China}
\email{lilong1@shanghaitech.edu.cn}

\begin{document}
\maketitle 

\begin{abstract}
The aim of this paper is to study 
the residual Monge-Amp\`{e}re mass of a plurisubharmonic function
with isolated singularity at the origin in $\bC^2$. 
We prove that the residual mass is zero 
if its Lelong number is zero at the origin, 
provided that it is $S^1$-invariant. 
 This result answers 
the zero mass conjecture raised by Guedj and Rashkovskii in this special case.
More generally, 
we obtain an estimate on the residual mass 
by the maximal directional Lelong number and Lelong number at the origin.

\end{abstract}

\section{Introduction}
Let $D$ be a bounded domain in $\bC^n$, 
and $u$ a $C^2$-continuous plurisubharmonic function on $D$. 
Then the Monge-Amp\`ere operator operates on $u$
and equals the following positive measure as 
\begin{equation}
\label{int-001}
\mbox{MA}(u): = (dd^c u)^n \geq 0,  
\end{equation}
where $d: = \d + \dbar$ 
and $d^c: = \frac{i}{2}(\dbar - \d)$.
This operator has great importance in pluripotential theory. 
However, it is fully non-linear and can not be defined for 
all plurisubharmonic functions on $D$, cf. \cite{Ceg86}, \cite{Kis84} and \cite{Siu75}.

On the other hand, there are several ways to define 
the Monge-Amp\`ere measure
for a plurisubharmonic function $u$, if extra conditions have been assumed. 
For instance, 
Bedford and Talyor \cite{BT} have shown that $(dd^c u)^n$ is well defined, 
if $u$ is further in $L^{\infty}_{loc}(D)$. 
Later Demailly \cite{Dem93} extended this definition 
to all plurisubharmonic functions whose unbounded locus are relatively compact in $D$. 
In particular, the operator $(dd^c)^n$ acts well 
on plurisubharmonic functions with isolated singularity. 

For simplicity, we take $D$ as the unit ball $B_1$ in $\bC^n$. 
Let $u$ be a plurisubharmonic function on $B_1$ that is locally bounded outside the origin. 
Then Guedj and Rashkovskii (Question 7, \cite{GZ15})
raised the following question: 
\begin{conj}
\label{conj-001}
Assume that $(dd^c u)^n$ has a Dirac mass at the origin. 
Does it imply that $u$ has a positive Lelong number at the origin\ \emph{?}  
\end{conj}
The atomic mass of $(dd^c u)^n$ at the origin is called the \emph{residual Monge-Amp\`ere mass} of $u$,  
and we can write it as 
\begin{equation}
\label{int-002}
\tau_u(0): = \frac{1}{\pi^n}\mbox{MA}(u) (\{ 0 \}).
\end{equation}
Here the normalization is chosen in such a way that we have  
$\tau_{\log|z|}(0) =1$.  
Denote the Lelong number of $u$ at the origin by $\nu_u(0)$. 
Then the above Conjecture (\ref{conj-001}) can be rephrased as follows.  
\begin{conj}
\label{conj-002}
$\nu_u(0)=0 \Rightarrow \tau_u(0) =0\ ?$
\end{conj}
For this reason, this problem is also called the \emph{zero mass conjecture}
for plurisubharmonic functions. 
In history, there have been many works that contribute to this problem, 
cf. \cite{Ceg02}, \cite{Ra01}, \cite{Ra13}, \cite{KR21}, \cite{BFJ07} and \cite{G10}. 
In particular, Rashkovskii \cite{Ra01} confirmed 
this conjecture, provided with toric symmetry on $u$. 
In this paper, we will study a more general symmetry called circular symmetry
for plurisubharmonic functions. 

Let $(z_1,\cdots, z_n)$ be the complex Euclidean coordinate on $\bC^n$, 
and then there is a natural $S^1$-action on it as 
\begin{equation}
\label{int-003}
z\rightarrow e^{i\theta}z: = (e^{i\theta}z_1,\cdots,e^{i\theta}z_n),  
\end{equation}
for all $\theta\in\bR$. 
A domain is balanced if it is invariant under this $S^1$-action. 
We say that a function $u$ on a balanced domain $D$ 
is circular symmetric, or $S^1$-invariant if
$$ u(e^{i\theta}z) = u(z),$$
for all $z\in D$. 
Then it is apparent that any toric symmetric function is also $S^1$-invariant. 

In fact, there is a deep connection between $S^1$-invariant plurisubharmonic 
functions and \emph{the Schwarz symmetrization} technique in classical analysis.
Berman and Berndtsson \cite{BB} proved that 
the Schwarz symmetrization of any $S^1$-invariant 
plurisubharmonic function is also plurisubharmonic. 
Moreover, the Lelong number at the origin 
is always increasing under this symmetrization, cf. \cite{Li19}.

On the other hand, this $S^1$-action is highly related to
\emph{the Hopf-fiberation} of the unit sphere $S^{2n-1}$ in $\bR^{2n} \cong \bC^n$, cf. \cite{GWZ86}. 
As the first attempt to utilize this geometric picture, 
we will restrict to $\bC^2$ in this paper, 
where the structure of the Hopf-fiberation 
$ S^1\hookrightarrow S^3\xrightarrow{p} S^2 $
is fully understood. 

For this reason, 
the domain $D$ is assumed to be the unit ball $B_1\subset \bC^2$ from now on. 
Then we introduce the family 
$\cF(B_1)$ (Definition (\ref{pre-def-001}))
as a collection  
of all $S^1$-invariant plurisubharmonic functions on $B_1$ that is $L^{\infty}_{loc}$ outside the origin. 
In order to perform calculus, 
we further introduce the family 
$\cF^{\infty}(B_1)$  (Definition (\ref{pre-def-002}))
as a sub-collection of $\cF(B_1)$
that is $C^2$-continuous outside the origin, 
and then we first confirm the zero mass conjecture for this family. 

\begin{theorem}[Theorem (\ref{thm-001})]
\label{thm-int-001}
For any function $u\in \cF^{\infty}(B_1)$, we have 
$$ \nu_u(0)=0 \Rightarrow \tau_u(0) =0. $$
\end{theorem}

The key observation is 
a decomposition formula (Theorem (\ref{thm-cal-001})) for the 
complex Monge-Amp\`ere mass. 
It decomposes the measure $(dd^c u)^2$ 
on the ball $B_R$ into two integrals on the boundary $S_R: = \d B_R$.  
The first integral corresponds to the so called 
\emph{pluricomplex energy} on $\bC\bP^1$ (Section (\ref{sub-001})),
and the second integral is a kind of $L^2$-Lelong number (Section (\ref{sub})). 
 
In fact, this decomposition formula 
implies a stronger result. 
For a general function $u\in \cF^{\infty}(B_1)$
without zero Lelong number condition, 
we obtain the following estimate 
on the residual Monge-Amp\`{e}re mass. 

\begin{theorem}[Theorem (\ref{thm-dn-001})]
\label{thm-int-0015}
For any function $u\in \cF^{\infty}(B_1)$, we have 
$$ [\nu_u(0)]^2 \leq \tau_{u}(0) \leq 2\lambda_u(0)\cdot \nu_u(0) + [\nu_u(0)]^2. $$
\end{theorem}
The one side inequality $[\nu_u(0)]^2 \leq \tau_{u}(0)$
was indicated by Cegrell, cf. \cite{Ceg04}, 
and it holds for a general function $u\in\cF(B_1)$.
The other side follows from 
our Lemma (\ref{lem-dn-000}). 
Here $\lambda_{u}(0)$ is the so called \emph{ maximal directional Lelong number}
of $u$ at the origin, cf. Definition (\ref{defn-dn-002}). 
It is actually the decreasing limit of $M_A(u)$ as $A\rightarrow +\infty$, 
where the constant $M_A(u)$ is 
the \emph{maximal directional Lelong number} of $u$ at the distance $A$, cf. Definition (\ref{defn-dn-001}). 
For a function $u\in \cF^{\infty}(B_1)$,
it is apparent that they are all non-negative real numbers.  
However, there is no a priori reason that 
they are finite for a general plurisubharmonic function. 

Surprisingly, 
the finiteness of $\lambda_u(0)$ and $M_A(u)$ are always true 
for a function $u\in \cF(B_1)$, cf. Lemma (\ref{lem-gc-001}). 
This is due to the following two facts: 
first, the function $u$ is bounded from above on the ball $B_1$
and from below on each boundary sphere $S_R$ with $R\in (0,1)$; 
second, the restriction of $u$ to each complex line through 
the origin is actually a convex function in the variable $t: = \log r\in (-\infty, 0)$. 
Therefore, the slope of the tangent line of this convex function 
can not grow arbitrarily large, 
when the angle of the line varies on $\bC\bP^1$. 

Then the idea is to utilize a regularization sequence $u_{\ep}\in \cF^{\infty}(B_1)$
to approximate the function $u\in \cF(B_1)$. 
The key step is to obtain a uniform estimate on the maximal directional Lelong number $M_B(u_{\ep})$ by $M_A(u)$ 
 for any constant $B>A>1$, cf. Lemma (\ref{lem-app-001}).
 These facts enable us to generalize Theorem (\ref{thm-int-001}) and (\ref{thm-int-0015}) as follows.

\begin{theorem}[Theorem (\ref{thm-002})]
\label{thm-int-002}
For any function $u\in \cF(B_1)$, 
we have 
$$ \nu_u(0)=0 \Rightarrow \tau_u(0) =0. $$
\end{theorem}

\begin{theorem}[Theorem (\ref{thm-pln-001})]
\label{thm-int-0025}
For any function $u\in \cF(B_1)$, we have 
$$ [\nu_u(0)]^2 \leq  \tau_{u}(0) \leq 2\lambda_u(0)\cdot \nu_u(0) + [\nu_u(0)]^2. $$
\end{theorem}

In order to illustrate the inequality in the above Theorem,
we compute several cases, 
including Demailly's example \cite{Dem93}
and Chi Li's construction \cite{Chi21},  
at the end of Section (\ref{sec02}). 
Moreover, we provide an counter-example of 
this inequality, when the function $u$ is no longer $S^1$-invariant. 

There is another point of view to look at a function $u\in \cF(B_1)$
via a variational approach. 
First we recall a few basic facts in K\"ahler geometry, 
cf. \cite{Sem74}, \cite{Don97}, \cite{C00} and \cite{Li22}. 
Consider a sub-geodesic ray in the space of K\"ahler potentials on $\bC\bP^1$. 
It is actually a local plurisubharmonic function $u$ on the product space 
$\bD^* \times \bC\bP^1$ that is $S^1$-invariant in the argument direction of $\bD^*$. 
Moreover, it is a geodesic ray 
if the following \emph{homogeneous complex Monge-Amp\`ere equation} holds 
\begin{equation}
\label{int-004}
(dd^c u)^2 = 0,
\end{equation}
on the product $\bD^* \times \bC\bP^1$. 

On the other hand, 
we note that the punctured disk $\bD^*$ acts on $B_1^*\subset \bC^2$ 
in a natural way. 
Then the punctured ball $B_1^*$ 
can be thought of as a non-trivial 
$\bD^*$-fiberation over $\bC\bP^1$, 
i.e. we have the following fiber bundle structure 
\begin{equation}
\label{int-005}
\bD^* \hookrightarrow B_1^* \xrightarrow{p} \bC\bP^1. 
\end{equation}

Comparing with the manifold $\bD^*\times \bC\bP^1$,
we have a simpler total space since the Euclidean metric on $B_1^*$ is flat. 
However, the fiberation structure 
corresponds to the Hopf-fiberation that is more complicated. 
In particular, the usual complex structure 
on $B_1^*\subset\bC^2$ is no longer 
a product of the complex structures on $\bD^*$ and $\bC\bP^1$.

In this way, a function $u\in\cF(B_1)$
can be viewed as a sub-geodesic ray 
on this non-trivial $\bD^*$-bundle (Definition (\ref{def-cf-001})), 
and it is a geodesic ray on this bundle if equation (\ref{int-004}) holds on $B_1^*$. 

This observation leads us to a 
new understanding about the decomposition formula
and the zero mass conjecture. 
In fact, the decomposition formula 
can be fit into an energy picture as shown in Theorem (\ref{thm-en-001}).
There a sub-geodesic ray corresponds to a convex energy functional 
on $(-\infty, 0)$, 
and geodesic rays are exactly the affine ones. 
The zero mass conjecture has also been rephrased  
under this picture, 
and it describes simultaneous 
zero asymptotic behaviors of two energy functionals as in Theorem (\ref{thm-en-002}).

Finally, we would like to point out 
that the decomposition formula 
(Theorem (\ref{thm-cal-001}))  
is very likely to be generalized to all dimensions. 
Then 
the zero mass conjecture 
can be proved in $\bC^n$ for all $S^1$-invariant plurisubharmonic functions in a similar manner. 
Moreover, an estimate on the residual mass as in Theorem (\ref{thm-int-0025}) may also be obtained
in higher dimensions.

\bigskip
\bigskip
\bigskip

\textbf{Acknowledgment: }
The author is very grateful 
to Prof. Xiuxiong Chen and Prof. Mihai P\u aun
for their continuous support and encouragement in mathematics. 
This problem has been raised to the author when he was studying in Fourier Institute, Grenoble. 
It is also a great pleasure to thank Chengjian Yao, Xiaojun Wu, Jian Wang, Wei Sun for lots of 
useful discussions.

Finally, the author wishes to thank Prof. Berndtsson, Prof. Rashkovskii, Prof. Xiangyu Zhou and Prof. Chi Li
for their valuable suggestions on the first version of this paper.


\bigskip

\section{ Plurisubharmonic functions with isolated singularity}

Denote $z: = (z_1, z_2)$ by the complex Euclidean coordinate on $\bC^2$. 
There is a natural $S^1$-action on it as  
$$ z \rightarrow e^{ i\theta } z,  $$
where $e^{i\theta} z: = (e^{i\theta} z_1, e^{i\theta} z_2)$, and $\theta$ is an arbitrary real number. 
We say that a domain $D$ is balanced if it is invariant under this action.
Moreover, a function $u$ on a balanced domain is said to be $S^1$-invariant if for all $z\in D$
$$u(e^{i\theta }z ) = u(z). $$ 

Assume that the origin $0\in \bC^2$ is contained in a balanced domain $D$, 
and we denote $D^*$ by the set $D - \{ 0 \}$. 
Consider a plurisubharmonic function $u$ on $D$,
and we adapt to the following definition. 

\begin{defn}
\label{pre-def-001}
A plurisubharmonic function $u$ on $D$ belongs to the family 
$\cF(D)$, if it is $S^1$-invariant and $L_{loc}^{\infty}$ on $D^*$. 
\end{defn}

We say that $u$ has an isolated singularity at the origin, if $u\in \cF(D)$ and $u (0) = -\infty$.  
In order to preform calculus,
we also introduce the following collection of functions with better regularities.


\begin{defn}
\label{pre-def-002}
A plurisubharmonic function $u$ on $D$ belongs to the family 
$\cF^{\infty}(D)$, if it is $S^1$-invariant and $C^2$-continuous on $D^*$.


\end{defn}

We note that a function $u\in \cF^{\infty}(D)$ also belongs to the family $\cF(D)$.
By shrinking $D$ to a smaller balanced domain if necessary, 
a function $u\in \cF(D)$, or $\cF^{\infty}(D)$ 
always has an upper bound. 
After adjusting a constant, we can further assume the following normalization condition 
$$\sup_D u \leq -1, $$
for all $u\in \cF(D)$, or $\cF^{\infty}(D)$.

\subsection{The residual mass}
In the following, 
we will focus on the local behavior of a plurisubharmonic function $u$
near the origin. 
To this purpose, 
it is enough to consider the balanced domain $D$
as a small ball centered at the origin.

Let $B_R\subset\bC^2$ be the open ball  with radius $R$ centered at the origin, 
and $B^*_R: = B_R - \{0\}$ be the corresponding punctured ball. 
Denote its boundary as $S_R: = \d B_R$,
and then $S_R$ is actually a $3$-sphere in $\bR^4$. 

Obviously, these two domains $B_R$, $B^*_R$ are balanced for each $R>0$, 
and then we can consider plurisubharmonic functions in the family $ \cF(B_R)$ and $\cF^{\infty}(B_R)$.

Thanks to Demailly's work \cite{Dem93}, 
the complex Monge-Amp\`ere measure 
of $u\in\cF(B_1)$ or $\cF^{\infty}(B_1)$
is well defined, namely, the following wedge product is  a bidegree-$(2,2)$ closed positive current 
$$  \text{MA}(u): = dd^c u \wedge dd^c u, $$
and then it is also a positive Borel measure on $B_1$. 
(Here we have used 
$dd^c = i \d\dbar$). 
Fixing an $R\in(0,1)$, we take this measure on the ball as 
\begin{equation}
\label{rm-000}
 \text{MA}(u)(B_R): = \int_{B_R} (dd^c u)^2 = \int \chi_{B_R} (dd^c u)^2,
\end{equation}
where $\chi_{B_R}$ is the characteristic function of the ball $B_R$. 
Then it builds a decreasing sequence of non-negative real numbers as $R\rightarrow 0$.
Thanks to the dominated convergence theorem,
this limit is exactly the residual Monge-Amp\`ere mass of $u$ at the origin, i.e. we have  
\begin{equation}
\label{rm-001}
 \tau_{u} (0) = \frac{1}{\pi^2} \lim_{R\rightarrow 0}  \text{MA}(u)(B_R). 
\end{equation}

In order to calculate the measure in equation (\ref{rm-000}), 
we first observe the following analogue of the Portemanteau theorem. 

\begin{lemma}
\label{lem-rm-001}
Suppose $u_j$ is a sequence of smooth plurisubharmonic functions on $B_1$, decreasing to $u\in\cF^{\infty}(B_1)$. 
Then we have 
\begin{equation}
\label{rm-002}
\emph{MA}(u) (B_R)= \lim_{j\rightarrow +\infty} \emph{MA}(u_j)(B_R),
\end{equation}
for all $R\in (0,1)$. 
\end{lemma}

\begin{proof}
It is enough to prove the following two inequalities. 
First, we claim  
\begin{equation}
\label{rm-003}
\text{MA}(u)(\overline{B}_R) \geq  \limsup_{j\rightarrow +\infty} \text{MA}(u_j)(\overline{B}_R),
\end{equation}
on any closed ball $\overline{B}_R$ in $B_1$. 
Second, we claim 
\begin{equation}
\label{rm-004}
\text{MA}(u)(B_R) \leq  \liminf_{j\rightarrow +\infty} \text{MA}(u_j)(B_R),
\end{equation}
on any open ball $B_R \subsetneq B_1$. 
Then we have 
\begin{equation}
\label{rm-0045}
 \text{MA}(u)(\overline{B}_R) =  \text{MA}(u)(B_R),
\end{equation}
since $u$ is $C^2$-continuous near the boundary $S_R$. 
Hence our result follows from equation (\ref{rm-003}) and (\ref{rm-004}).  

To prove the first claim, we observe that 
there exists a sequence of smooth cut off functions $\chi_k$,
such that $\chi_k = 1$ on $\overline{B}_R$ and $\chi_k = 0$ outside of $B_{R+ \frac{1}{k}}$. 
Therefore, we have for a fixed $k$
\begin{eqnarray}
\label{rm-005}
\limsup_{j\rightarrow +\infty} \text{MA}(u_j)(\overline{B}_R) &=& \limsup_{j\rightarrow +\infty} \int \chi_{\overline{B}_R}(dd^c u_j)^2
\nonumber\\
&\leq &  \limsup_{j\rightarrow +\infty} \int \chi_{k}(dd^c u_j)^2
\nonumber\\
&= & \int \chi_k (dd^c u)^2 \leq \text{MA}(u)(B_{R+\frac{1}{k}}).  
\end{eqnarray}
The equality on the last line of the above equation
follows from the convergence $(dd^c u_j)^2 \rightarrow (dd^c u)^2$ in the sense of currents, cf. \cite{Dem93}.  

Finally, the inequality (equation (\ref{rm-003})) follows by taking $k\rightarrow +\infty$ in equation (\ref{rm-005}).
The second claim (equation (\ref{rm-004})) can also be proved in a similar way.

\end{proof}

\begin{rem}
\label{rem-rm-001}
For a general $u\in \cF(B_1)$, Lemma (\ref{lem-rm-001}) may fail to be true for all $R\in (0,1)$.  
However, if the Monge-Amp\`ere measure of $u$ has no mass on the boundary of a ball, 
namely, we assume  
$$  \int_{S_R} (dd^c u)^2 = 0, $$
for a fixed $R$, then equation (\ref{rm-0045}) still holds, 
and the convergence (equation (\ref{rm-002})) follows from the same argument. 
\end{rem}

Another advantage to deal with the family $\cF^{\infty}(B_R)$
is that we can perform integration by parts on the current $(dd^c u)^2$ as follows.

\begin{prop}
\label{prop-rm-001}
For a plurisubharmonic function $u\in\cF^{\infty}(B_1)$, 
we have 
\begin{equation}
\label{pre-001}
\int_{B_R} (dd^c u)^2 = \int_{S_R} d^c u \wedge dd^c u,
\end{equation}
for all $R\in (0,1)$. 

\end{prop}

\begin{proof}
Let $\rho(z): = \rho (|z|) \in C^{\infty}_0 (\bC^2)$ be a non-negative function,
satisfying $\rho(r) = 0$ for $r\geq 1$, and $\int_{\bC^2}\rho\ d\lambda= 1$. 
For each $\ep>0$ small, we rescale it as 
$$ \rho_{\ep} (z): = \ep^{-4} \rho(z/\ep ). $$
Then we have the following standard regularization of $u\in \cF^{\infty}(B_1)$ by convolution,
namely, we have on any closed ball contained in $B_1$ 
\begin{eqnarray}
\label{rm-006}
u_{\ep}(z): &=& (u \ast \rho_{\ep})(z) 
\nonumber\\
&=& \int_{|z-y|\leq \ep} \rho_{\ep} (z-y) u(y) d\lambda(y)
\nonumber\\
&=& \int_{|w|\leq 1} u(z- \ep w) \rho(w) d\lambda(w). 
\end{eqnarray}
Hence $u_{\ep}(z)$ is a sequence of smooth plurisubharmonic functions, decreasing to $u(z)$ as $\ep\rightarrow 0$. 
Moreover, we have 
$u_{\ep}\rightarrow u$ uniformly in $C^2$-norm on any compact subset in $B^*_1$. 

Due to Stokes' theorem, we compute  
\begin{equation}
\label{rm-007}
\int_{B_R} (dd^c u_{\ep})^2 = \int_{S_R} d^c u_{\ep} \wedge dd^c u_{\ep},
\end{equation}
on any ball $B_R\subsetneq B_1$. 
Thanks to Lemma (\ref{lem-rm-001}), 
The L.H.S. of equation (\ref{rm-007}) converges to 
the measure $\text{MA}(u)(B_R)$. 
Moreover, 
the R.H.S. of this equation converges to the desired integral
$$  \int_{S_R} d^c u \wedge dd^c u, $$
since $u_{\ep}\rightarrow u$ uniformly in $C^2$-norm near the sphere $S_R$. 
Then our result follows. 
\end{proof}

Take an $S^1$-action on the regularization $u_{\ep}(z)$, we note that 
it is also invariant under this action. 
In fact, if put $w' = e^{-i\theta} w$,  then we have  
\begin{eqnarray}
u_{\ep}(e^{i\theta}z)& = & \int u(e^{i\theta}z- \ep w) \rho(w) d\lambda(w)
 \nonumber\\
 &= & \int u\{ e^{i\theta}(z- \ep w') \}  \rho(w') d\lambda(w')
 \nonumber\\
 &=& \int u (z- \ep w')\rho(w') d\lambda(w')
 = u_{\ep}(z).  
\end{eqnarray}
Therefore, the following result holds.  
\begin{cor}
\label{cor-rm-001}
For any function $u\in \cF(B_1)$, 
there exists a sequence of $S^1$-invariant smooth plurisubharmonic function $u_j$ decreasing pointwise to $u$,
possible on a slightly smaller ball. 
\end{cor}

In order to illustrate the use of Proposition (\ref{prop-rm-001}), 
let us consider a simpler case. 
Suppose $u$ is a subharmonic function on the unit disk $\bD\subset \bC$, 
which is also $C^2$-continuous on $\bD^*$. 
Then equation (\ref{pre-001}) reduces to 

\begin{equation}
\label{pre-002}
\int_{ |z|< R } dd^c u = \int_{|z| =R} d^c u,
\end{equation}
for all $R\in (0,1)$. 

Utilizing the polar coordinate $z = re^{i\theta}$ on $\bC^*$, 
we have the computation
 
\begin{equation}
\label{rm-008}
d^c u = -\Im (\dbar u) = \frac{1}{2} \left\{  (r \d_r u ) d\theta - ( r^{-1} \d_{\theta} u) dr \right\}.
\end{equation}
Denote the circular average of $u$ by
$$ \hat{u}(r) : = \frac{1}{2\pi} \int_0^{2\pi} u(r e^{i\theta}) d\theta, $$
Then it follows 
$$ \frac{1}{\pi} \int_{|z|< R} dd^c u =  r \d_r \hat u |_{r =R}\rightarrow \nu_u(0), $$
as $R\rightarrow 0$. 
Therefore, we prove that the residual mass of $u$ at the origin 
is zero if its Lelong number is in $\bC$.  

\section{The Hopf-coordinates }
In this section, we are going to compute the integral on the 
R.H.S. of equation (\ref{pre-001}). 
It boils down to calculate the following $3$-form on the $3$-sphere 
\begin{equation}
\label{hp-001}
 d^c u \wedge dd^c u|_{S_R}, 
 \end{equation}
for $u\in \cF^{\infty}(B_1)$, and $R\in(0,1)$. 
First, we note that the $S^1$-action 
also performs on this $3$-sphere, 
and it induces the Hopf fiberation. 
The Hopf fiberation $S^1\hookrightarrow S^3\xrightarrow{p} S^2$
is an example of non-trivial $S^1$-fiber bundle over $S^2$. 
It can be illustrated via the following real coordinates. 

Write the unit $3$-sphere as 
$$ S^3: = \{ x^2_1 + y_1^2 + x_2^2+ y_2^2 =1 \},$$
where $(x_1, y_1, x_2, y_2)\in\bR^4$ is the Euclidean coordinate.  
Let  $\theta\in [0, \pi], \vp\in [0, 2\pi]$ be the spherical coordinate of the unit $2$-sphere $S^2\subset \bR^3$. 
Define the following coordinate for $\eta\in [0, 4\pi]$  
$$
x_1 = \cos\left( \frac{\eta + \vp}{2} \right) \sin \left( \frac{\theta}{2} \right),
\ \ \ \ 
y_1 = \sin \left( \frac{\eta + \vp}{2} \right) \sin \left( \frac{\theta}{2} \right),
$$
$$
x_2 = \cos\left( \frac{\eta - \vp}{2} \right) \cos \left( \frac{\theta}{2} \right),
 \ \ \ \ 
 y_2 = \sin \left( \frac{\eta - \vp}{2} \right) \cos \left( \frac{\theta}{2} \right). 
$$
It is clear that $\eta$ is the direction under the $S^1$-action, 
and the Hopf fiberation $p: S^3 \rightarrow S^2$ is the submersion
\begin{equation}
\label{hp-0010}
\left(  2 (x_1x_2 + y_1y_2), \  2(  x_2y_1 - x_1y_2), \  x_2^2 + y_2^2 - x_1^2-y_1^2 \right). 
\end{equation}

This leads us to  introduce the following \emph{real Hopf-coordinate} 
$$ (r, \eta, \theta, \vp)$$ 
for all $r\in \bR_+$, $\theta\in [0,\pi]$ and $\eta,  \vp \in \bR$ 
to represent a point in  $\bR^4 - \{ 0 \}$. 
Then the change of variables is 
$$
x_1 = r\cos\left( \frac{\eta + \vp}{2} \right) \sin \left( \frac{\theta}{2} \right),
\ \ \ \ 
y_1 = r\sin \left( \frac{\eta + \vp}{2} \right) \sin \left( \frac{\theta}{2} \right),
$$
$$
x_2 = r\cos\left( \frac{\eta - \vp}{2} \right) \cos \left( \frac{\theta}{2} \right),
 \ \ \ \ 
 y_2 = r\sin \left( \frac{\eta - \vp}{2} \right) \cos \left( \frac{\theta}{2} \right). 
$$
It follows that a $3$-sphere $S_R$ as the boundary of the ball $B_R$
can be written as
$$ S_R: =  \{ x^2_1 + y_1^2 + x_2^2+ y_2^2 = R^2 \} \subset \bR^4, $$
for a fixed $R>0$.
Moreover, if the angle $\vp$ varies in $[0, 2\pi)$ and $\eta$ in $[0, 4\pi)$,
then this coordinate runs over all the points on the $3$-sphere $S_R$ exactly once. 






\subsection{Complex version}
There is another way to view the Hopf fiberation 
through the complex coordinates. 
Let $z:= (z_1, z_2)$ be the complex Euclidean coordinate on $\bC^2$, 
and then the unit $3$-sphere $S^3$ can be characterized as 
$$ S^3: = \{ |z_1|^2 + |z_2|^2 =1 \}. $$
Topologically, the $2$-sphere $S^2$
can be identified with the extended complex plane $\bC_{\infty}: =\bC \bigcup \{ \infty \}$
via the stereographic projection. 
Moreover, the complex projective line $\bC\bP^1$ can also be identified with $\bC_{\infty}$
via the continuous map $f: S^3 \rightarrow S^2$ as 
$$f: (z_1, z_2) \rightarrow \frac{z_1}{z_2}\in\bC_{\infty}.$$
Then we can define the fiber map $p$ in the fiber bundle 
$S^1\hookrightarrow S^3\xrightarrow{p} S^2$ as 
$$ p: (z_1, z_2) \rightarrow [z_1: z_2] \in \bC\bP^1, $$
and the pre-image of each point in $\bC\bP^1$ is a great circle in $S^3$. 

In order to illustrate the idea, 
we first introduce the following easier version. 
It is well known that $\bC\bP^1$ can  be covered by 
two holomorphic coordinate charts, consisting of 
$$U_1: = \bC\bP^1 - [1:0], \ \ \ U_2: =  \bC\bP^1- [0: 1], $$
where $U_1$ is identified with $\bC_{\infty} - \{\infty \}$
and $U_2$ with $\bC_{\infty} - \{0\}$. 
Write their corresponding holomorphic coordinates as $\z: = z_1/ z_2$ and $\xi = \z^{-1}$.
Then we give the following two homeomorphisms as 
the local trivializations of the fiber map $p$: 
define 
$\psi_1: \bC\times S^1 \rightarrow p^{-1} (\bC)$ and 
$\psi_2: (\bC_{\infty} - \{0 \} ) \times S^1  \rightarrow p^{-1} (\bC_{\infty} - \{0 \} )$ 
as 
\begin{equation}
\label{hp-0011}
\psi_1(\z, \eta' ) = \left(  \frac{\z e^{i\eta'}}{(1+ |\z|^2 )^{1/2}},\ \frac{e^{i\eta'}}{(1+ |\z|^2)^{1/2}}     \right),
\end{equation}
and 
\begin{equation}
\label{hp-0012}
\psi_2(\xi, \eta') = \left(  \frac{ e^{i\eta'}}{(1+ |\xi|^{2} )^{1/2}},\  \frac{ \xi e^{i\eta'}}{(1+ |\xi|^2)^{1/2}}     \right),
\end{equation}
for all $\z, \xi \in \bC$ and $\eta' \in \bR$. 
It is apparent that we have 
$p\circ \psi_1 (\z, \eta') = [\z: 1]$ on $U_1$ and $p \circ \psi_2 (\xi, \eta' ) = [1: \xi] $ on $U_2$.
Then these two local trivializations describe the Hopf fiberation in one way.

\subsection{Another version}
Next, 
we can write the complex variables of $\bC^2$ in terms of the real Hopf-coordinate as 
\begin{equation}
\label{hp-0013}
z_1 =  (r\sin (\theta /2))e^{\frac{i }{2}(\eta + \vp) }, \ \ \ z_2 = (r\cos (\theta /2))e^{\frac{i}{2} (\eta - \vp)}, 
\end{equation}
for all $r>0$, $\theta\in[0,\pi]$, $\vp\in[0,2\pi]$ and $\eta\in [0,4\pi]$. 
Moreover, the complex variable on $\bC_{\infty}$  is 
\begin{equation}
\label{hp-00135}
 \z: = \frac{z_1}{z_2} = \tan(\theta/2) e^{i\vp}.  
\end{equation}
and we obtain the following change of variables
\begin{equation}
\label{hp-0014}
z_1 = r e^{\frac{i}{2}\eta} \frac{(\z\cdot |\z| )^{1/2}}{(1+ |\z|^2)^{1/2}},\ \ \ 
z_2 =  r e^{\frac{i}{2}\eta}  \frac{(\bar\z / |\z|)^{1/2}}{(1+ |\z|^2)^{1/2}}. 
\end{equation}
In fact, they can be viewed as different local trivializations of the Hopf-fiberation with the fiber map $p$. 
Denote  $\l_+, \l_-$ by the following half circles on $S^2\cong \bC\bP^1$ 
$$\l_+: =  \left\{ [t : 1] \in \bC\bP^1; \ \  t\in [0, +\infty]  \right\}; $$
$$\l_-: =  \left\{ [t : 1] \in \bC\bP^1; \ \  t\in [ -\infty, 0 ]  \right\}, $$
and define two holomorphic coordinate charts as 
$$V_1: = S^2- \l_+; \ \ \ V_2: = S^2 - \l_-. $$
Similarly, we have another two half circles as 
$$\jmath_{+}: =  \left\{ [ e^{i\theta'}: 1] \in \bC\bP^1; \ \  \theta' \in [ 0, \pi ]  \right\};$$
$$\jmath_{-}: =  \left\{ [ e^{i\theta'}: 1] \in \bC\bP^1; \ \  \theta' \in [ \pi, 2\pi ]  \right\},$$
and another two charts are defined as 
$$V_3: = S^2- \jmath_+; \ \ \ V_4: = S^2 - \jmath_-. $$
It is apparent that each of the charts $V_i, i=1,2,3,4$ 
can be identified to the slit plane $\bC_+$ (or $\bC_-$)
via the stereographic projections, 
and they together
cover the whole sphere $S^2$. 
Then we can introduce homeomorphisms $\psi'_i$ between $V_i\times S^1$ and $p^{-1}(V_i)$ 
for all $i=1,2,3,4$, 
and they will build the local trivializations for the fiber map $p$ in a different way. 

In particular, equation (\ref{hp-0014}) gives the map $\psi'_1: V_1\times S^1 \rightarrow p^{-1}(V_1)$ as 
\begin{equation}
\label{hp-0015}
\psi'_1(\z, \eta) : =  \left( e^{\frac{i}{2}\eta} \frac{(\z\cdot |\z| )^{1/2}}{(1+ |\z|^2)^{1/2}},\ \ \ 
  e^{\frac{i}{2}\eta}  \frac{( |\z| / \z )^{1/2}}{(1+ |\z|^2)^{1/2}} \right). 
\end{equation}
Therefore,
we introduce the following coordinate 
$$ (r, \eta, \z, \bar\z)$$
for all $r>0$, $\eta\in \bR$ and $\z\in\bC_{\infty}$ to represent a point in $ \bC^2 -\{ 0\}$,
and refer it as the \emph{complex Hopf-coordinate}.

We note that the coordinate $\z$ (under the trivialization $\psi'_1$) is no longer continuous across the line $\l_+$,
since its fractional power $\z^{1/2}$ is multi-valued. 
However, we do not worry about this problem if the function $u$ is $S^1$-invariant,
and the reason is as follows.

In fact, the trivialization $\psi'_2$ on $V_2\times S^1$ 
will take another analytic branch of the two-valued holomorphic function $\z^{1/2}$,
and then we can write 
\begin{equation}
\label{hp-0016}
\psi'_2(\z, \eta) : =  \left( e^{\frac{i}{2}\eta} \frac{ e^{i\pi}(\z\cdot |\z| )^{1/2}}{(1+ |\z|^2)^{1/2}},\ \ \ 
  e^{\frac{i}{2}\eta }  \frac{e^{-i\pi}( |\z| / \z )^{1/2}}{(1+ |\z|^2)^{1/2}} \right). 
\end{equation}
Hence it follows 
$ u\circ \psi_1 = u\circ \psi_2$ for all $(\z, \eta)$ in the overlapping area. 
In other words, the function $u$ is periodic in the angle $\vp$ direction with period $2\pi$(instead of $4\pi$!).

\begin{rem}
\label{rem-hp-001}
In fact, there is no difference between the two trivializations $\psi_1$ and $\psi'_1$
for any $S^1$-invariant function, since we can rewrite $\psi'_1$ as follows from equation (\ref{hp-00135})
\begin{equation}
\label{hp-0017}
\psi'_1(\z, \eta) : =  \left( e^{-\frac{i}{2} \vp} \frac{\z e^{\frac{i}{2} \eta}}{(1+ |\z|^2)^{1/2}},\ \ \ 
  e^{-\frac{i}{2}\vp }  \frac{ e^{\frac{i}{2}\eta}}{(1+ |\z|^2)^{1/2}} \right),
\end{equation}
and then it follows $u\circ \psi_1 = u\circ \psi'_1$ for all $(\z, \eta)$ in the overlapping area.

\end{rem}

\section{The decomposition formula}

Now we are going to perform local computations near a point $b\in S_R$,
under the complex Hopf-coordinate. 
It is noted that we will directly calculate on $\bar{z}^2_1$ and $\bar{z}^2_2$ in the following, 
and the fractional power $\z^{1/2}$ will not be used essentially.

\subsection{The $1$-form}
Recall that we have 
\begin{equation}
\label{com-000}
z_1 = r e^{\frac{i}{2}\eta} \frac{(\z\cdot |\z| )^{1/2}}{(1+ |\z|^2)^{1/2}},\ \ \ 
z_2 =  r e^{\frac{i}{2}\eta}  \frac{(\bar\z / |\z|)^{1/2}}{(1+ |\z|^2)^{1/2}},
\end{equation}
and then the following relations exist:
\begin{equation}
\label{com-001}
z_1 \cdot \bar z_2 = r^2 \frac{\z}{1+ |\z|^2},\ \ \ z_1\cdot z_2 = r^2 \frac{e^{i\eta}|\z|}{1 + |\z|^2}, 
\end{equation}
and 
\begin{equation}
\label{com-0011}
|z_1|^2 + |z_2|^2 = r^2,
\end{equation}
together with 
\begin{equation}
\label{com-002}
|z_2|^2 - |z_1|^2 = r^2 \frac{1- |\z|^2}{1+ |\z|^2} = r^2 \cos\theta.
\end{equation}
The first goal is to calculate the following $1$-form 
\begin{eqnarray}
\label{com-003}
d^c u &=& \frac{i}{2} (\dbar u - \d u )
=  - \Im (\dbar u)
\nonumber\\
&= & -\Im \left( \frac {\d u}{\d \bar z_1} d\bar z_1 \right)  -\Im \left( \frac {\d u}{\d \bar z_2} d\bar z_2 \right).
\end{eqnarray}
The first term on the R.H.S. of the above equation can be computed as
\begin{equation}
\label{com-004}
\bar z_1 =  r e^{- \frac{i}{2}\eta} \frac{(\bar\z\cdot |\z| )^{1/2}}{(1+ |\z|^2)^{1/2}}, \ \ \ 
\bar z^2_1 =  r^2 e^{- i\eta} \frac{\bar\z\cdot |\z| }{(1+ |\z|^2)},
\end{equation}
and it follows 
\begin{eqnarray}
\label{com-005}
2 \bar z_1 d\bar z_1 &=&  2 r e^{-i \eta} \frac{\bar\z\cdot |\z| }{(1+ |\z|^2)} dr  - i r^2 e^{-i \eta} \frac{\bar\z\cdot |\z| }{(1+ |\z|^2)} d\eta
\nonumber\\
&+&  r^2 e^{-i\eta} \left\{ \d_\z \left( \frac{\bar\z\cdot |\z| }{1+ |\z|^2} \right) d\z +  \d_{\bar\z} \left( \frac{\bar\z\cdot |\z| }{1+ |\z|^2} \right) d\bar\z\right\},
\end{eqnarray} 
and then we have 
\begin{equation}
\label{com-006}
\d_\z \left(   \frac{\bar\z\cdot |\z| }{1+ |\z|^2}  \right) = \frac{\bar\z^2}{2|\z|} \cdot \frac{1- |\z|^2}{(1+ |\z|^2)^2 },
\end{equation}
and 
\begin{equation}
\label{com-007}
\d_{\bar\z} \left(   \frac{\bar\z\cdot |\z| }{1+ |\z|^2}  \right) = \frac{|\z|}{2} \cdot \frac{3+ |\z|^2}{(1+ |\z|^2)^2 }.
\end{equation}
Combing equations (\ref{com-006}), (\ref{com-007}) with (\ref{com-005}), we obtain
\begin{eqnarray}
\label{com-008}
2 \bar z_1 d\bar z_1 &=&  2 r e^{-i \eta} \frac{\bar\z\cdot |\z| }{(1+ |\z|^2)} dr  - i r^2 e^{-i \eta} \frac{\bar\z\cdot |\z| }{(1+ |\z|^2)} d\eta
\nonumber\\
&+&  r^2 e^{-i\eta} \frac{\bar\z^2}{2|\z|} \cdot \frac{1- |\z|^2}{(1+ |\z|^2)^2 } d\z 
\nonumber\\
&+&   r^2 e^{-i\eta}  \frac{|\z|}{2} \cdot \frac{3+ |\z|^2}{(1+ |\z|^2)^2 }   d\bar\z,
\end{eqnarray} 
and it follows 
\begin{eqnarray}
\label{com-009}
4  \ d\bar z_1 &=&  4 e^{-\frac{i}{2} \eta} \frac{ (\bar\z\cdot |\z|)^{1/2} }{(1+ |\z|^2)^{1/2}} dr  - 2 i r e^{-\frac{i}{2} \eta} \frac{(\bar\z\cdot |\z| )^{1/2} }{(1+ |\z|^2)^{1/2}} d\eta
\nonumber\\
&+&  r e^{-\frac{i}{2}\eta} \left( \frac{\bar\z}{|\z|} \right)^{\frac{3}{2}} \cdot \frac{1- |\z|^2}{(1+ |\z|^2)^{3/2} } d\z 
\nonumber\\
&+&   r e^{-\frac{i}{2}\eta}  \left( \frac{\z}{|\z|} \right)^{\frac{1}{2}} \cdot \frac{3+ |\z|^2}{(1+ |\z|^2)^{3/2} }   d\bar\z. 
\end{eqnarray} 
In fact, we can divide $\bar{z}^2_1$ in equation (\ref{com-008}), and obtain 
\begin{eqnarray}
\label{com-010}
4  \ d\bar z_1 =   \bar z_1  \left\{ 4 r^{-1}dr  - 2 i  d\eta +  \frac{ 1 }{ \z }  \cdot \frac{1- |\z|^2}{(1+ |\z|^2) } d\z  +  \frac{1}{\bar\z} \cdot \frac{3+ |\z|^2}{(1+ |\z|^2) }   d\bar\z  \right\}.
\nonumber\\
\end{eqnarray} 
Similarly, we have 
\begin{equation}
\label{com-011}
\bar z_2 =  r e^{- \frac{i}{2}\eta} \frac{( \z / |\z| )^{1/2}}{(1+ |\z|^2)^{1/2}}, \ \ \ 
\bar z^2_2 =  r^2 e^{- i\eta} \frac{\z / |\z| }{(1+ |\z|^2)},
\end{equation}
and then it gives 
\begin{eqnarray}
\label{com-012}
2 \bar z_2 d\bar z_2 &=&  2 r e^{-i \eta} \frac{ \z / |\z| }{(1+ |\z|^2)} dr  - i r^2 e^{-i \eta} \frac{ \z / |\z| }{(1+ |\z|^2)} d\eta
\nonumber\\
&+&  r^2 e^{-i\eta} \left\{ \d_\z \left( \frac{ \z / |\z| }{1+ |\z|^2} \right) d\z +  \d_{\bar\z} \left( \frac{ \z / |\z| }{1+ |\z|^2} \right) d\bar\z\right\}.
\end{eqnarray} 
Moreover, we have 
\begin{equation}
\label{com-013}
\d_\z \left(   \frac{ \z / |\z| }{1+ |\z|^2}  \right) = \frac{1}{2|\z|} \cdot \frac{1- |\z|^2}{(1+ |\z|^2)^2 },
\end{equation}
and 
\begin{equation}
\label{com-014}
\d_{\bar\z} \left(   \frac{\z / |\z| }{1+ |\z|^2}  \right) = - \frac{\z^2}{2 |\z|^3} \cdot \frac{1 + 3|\z|^2}{(1+ |\z|^2)^2 }.
\end{equation}
It follows 
\begin{eqnarray}
\label{com-015}
2 \bar z_2 d\bar z_2 &=&  2 r e^{-i \eta} \frac{\z / |\z| }{(1+ |\z|^2)} dr  - i r^2 e^{-i \eta} \frac{\z /  |\z| }{(1+ |\z|^2)} d\eta
\nonumber\\
&+&  r^2 e^{-i\eta} \frac{1}{2|\z|} \cdot \frac{1- |\z|^2}{(1+ |\z|^2)^2 } d\z 
\nonumber\\
&- &   r^2 e^{-i\eta}  \frac{\z^2}{2|\z|^3} \cdot \frac{1+ 3|\z|^2}{(1+ |\z|^2)^2 }   d\bar\z,
\end{eqnarray} 
and we further simplify as
\begin{eqnarray}
\label{com-016}
4  \ d\bar z_2 =   \bar z_2  \left\{ 4 r^{-1}dr  - 2 i  d\eta +  \frac{ 1 }{ \z }  \cdot \frac{1- |\z|^2}{(1+ |\z|^2) } d\z  -  \frac{1}{\bar\z} \cdot \frac{ 1+ 3|\z|^2}{(1+ |\z|^2) }   d\bar\z  \right\}.
\nonumber\\
\end{eqnarray} 

\bigskip 

Next we are going to use the chain rule as follows 
$$ \frac{\d u}{ \d \bar z_1} =  \frac{\d u}{ \d r} \frac{\d r}{ \d \bar z_1} +   \frac{\d u}{ \d \z} \frac{\d \z}{ \d \bar z_1} + \frac{\d u}{ \d \bar\z} \frac{\d\bar\z}{ \d \bar z_1}, $$
and 
$$ \frac{\d u}{ \d \bar z_2} =  \frac{\d u}{ \d r} \frac{\d r}{ \d \bar z_2} +   \frac{\d u}{ \d \z} \frac{\d \z}{ \d \bar z_2} + \frac{\d u}{ \d \bar\z} \frac{\d\bar\z}{ \d \bar z_2}.  $$
Since $\z$ is a holomorphic function of $z$, 
it is clear that $\d \z / \d\bar z_1 = 0$, and $\d \z / \d\bar z_2 = 0$.
Moreover, it follows from equation (\ref{com-0011}) that $\d r / \d \bar z_1 = z_1 / 2r$ and $\d r / \d \bar z_2 = z_2 / 2r$.
Hence we have 
\begin{eqnarray}
\label{com-017}
&&4 \Im\left(  \frac{\d u}{ \d r} \frac{\d r}{ \d \bar z_1} d\bar z_1 \right)
\nonumber\\
&=& (r^{-1} \d_r u) |z_1|^2 \left\{  - d\eta +    \frac{1- |\z|^2}{ 2 (1+ |\z|^2) } \Im\left( \frac{ d\z }{ \z }  \right) +  
 \frac{3+ |\z|^2}{ 2 (1+ |\z|^2) }  \Im\left( \frac{d\bar\z}{\bar\z} \right) \right\}
\nonumber\\
&=&  -  (r^{-1} \d_r u) |z_1|^2 \left\{    d\eta +  \Im\left( \frac{ d\z }{ \z }    \right)  \right\},
\nonumber\\
\end{eqnarray}
where we used the equality $\Im(\z^{-1} d\z) = - \Im(\bar\z^{-1}d\bar\z)$. 
Similarly, it follows 
\begin{eqnarray}
\label{com-018}
&&4 \Im\left(  \frac{\d u}{ \d r} \frac{\d r}{ \d \bar z_2} d\bar z_2 \right)
\nonumber\\
&=& (r^{-1} \d_r u) |z_2|^2 \left\{  - d\eta +    \frac{1- |\z|^2}{ 2 (1+ |\z|^2) } \Im\left( \frac{ d\z }{ \z }  \right) -
 \frac{ 1 +3 |\z|^2}{ 2 (1+ |\z|^2) }  \Im\left( \frac{d\bar\z}{\bar\z} \right) \right\}
\nonumber\\
&=&  -  (r^{-1} \d_r u) |z_2|^2 \left\{    d\eta -  \Im\left( \frac{ d\z }{ \z }    \right)  \right\}.
\nonumber\\
\end{eqnarray}
Combing with equation (\ref{com-017}) and (\ref{com-018}), we further have 
\begin{eqnarray}
\label{com-019}
&& - 4 \Im\left(  \frac{\d u}{ \d r} \frac{\d r}{ \d \bar z_1} d\bar z_1 \right) - 4 \Im\left(  \frac{\d u}{ \d r} \frac{\d r}{ \d \bar z_2} d\bar z_2 \right)
\nonumber\\
&=& ( r\d_r u ) \left\{    d\eta    - \frac{ 1 - |\z|^2}{ 1 + |\z|^2 } \Im\left( \frac{ d\z }{ \z }  \right)    \right\}, 
\end{eqnarray}
and note that the R.H.S. of equation (\ref{com-019}) equals to 
\begin{equation}
\label{com-020}
 ( r\d_r u )  (  d\eta  - \cos\theta d\vp ) 
 \end{equation}
in the real Hopf-coordiante. 

Furthermore, it is straightforward to have 
$$ \frac{\d\bar\z}{\d \bar z_1} = \frac{1}{\bar z_2}, \ \ \  \frac{\d \bar\z}{\d \bar z_2} = - \frac{\bar z_1}{ ( \bar z_2)^2}. $$
Then we obtain 
\begin{eqnarray}
\label{com-021}
&& 4 \frac{\d u}{ \d \bar\z} \frac{\d\bar\z}{ \d \bar z_1} d\bar z_1 
\nonumber\\
&=& (\bar\z \d_{\bar\z}u ) \left\{  4 r^{-1} dr - 2i d\eta +  \frac{1- |\z|^2}{(1+ |\z|^2) } \left( \frac{ d\z }{ \z } \right)   +  \frac{3+ |\z|^2}{(1+ |\z|^2) }  \left(\frac{d\bar\z}{\bar\z} \right)    \right\},
\nonumber\\
\end{eqnarray}
and 
\begin{eqnarray}
\label{com-022}
&& 4 \frac{\d u}{ \d \bar\z} \frac{\d\bar\z}{ \d \bar z_2} d\bar z_2
\nonumber\\
&=& -(\bar\z \d_{\bar\z}u ) \left\{  4 r^{-1} dr - 2i d\eta +  \frac{1- |\z|^2}{(1+ |\z|^2) } \left( \frac{ d\z }{ \z } \right)   -  \frac{ 1+ 3 |\z|^2}{(1+ |\z|^2) }  \left(\frac{d\bar\z}{\bar\z} \right)    \right\}. 
\nonumber\\
\end{eqnarray}
Moreover, we note the first three terms on the R.H.S. of equation (\ref{com-021}) and (\ref{com-022}) 
are only differ by a minus sign. Hence it follows 

\begin{eqnarray}
\label{com-023}
&& - \Im \left( \frac{\d u}{ \d \bar\z} \frac{\d\bar\z}{ \d \bar z_1} d\bar z_1 +  \frac{\d u}{ \d \bar\z} \frac{\d\bar\z}{ \d \bar z_2} d\bar z_2 \right)
\nonumber\\
&=&  - \Im \left(\d_{\bar\z}u \cdot d\bar\z   \right) =  \Im \left(\d_{\z}u \cdot d \z   \right) . 
\end{eqnarray}
In conclusion, we obtain the following formula. 
\begin{lemma}
\label{lem-com-001}
For any $u\in \cF^{\infty}(B_1)$, we have 
\begin{eqnarray}
\label{com-024}
4\ d^c u &=& (r\d_r u) \left\{ d\eta -  \cos\theta \cdot \Im\left(  \frac{d\z}{\z} \right) \right\}  + 4 \Im (\d_{\z} u \cdot d\z )
\nonumber\\
&=& (r u_r)d\eta +( 2\sin\theta \cdot  u_{\theta}  - \cos\theta \cdot ru_r ) d\vp  - \frac{2 u_{\vp}}{\sin\theta} d\theta. 
\end{eqnarray}
where $\cos\theta = (1- |\z|^2)(1+ |\z|^2)^{-1}$. 
\end{lemma}
\begin{proof}
Combine equation (\ref{com-019}) and (\ref{com-023}), and then the equality follows in the complex Hopf-coordinate. 
By utilizing the real Hopf-coordinate, we obtain the second line on the R.H.S. of equation (\ref{com-024}), 
and one can check the following change of variables
$$2 \Im (\d_{\z} u \cdot d\z ) =  \sin\theta (\d_{\theta} u) d\vp -  \frac{ 1 }{\sin\theta} ( \d_{\vp}u)d\theta, $$
Then our result follows. 
\end{proof}

\subsection{The $3$-form} 
In this section, we continue our computation on the complex hessian $dd^c u$
for a $u\in \cF^{\infty}(B_1)$ in the complex Hopf-coordinate. Again, we perform the calculation near a point on $S_R$ for an $R\in (0,1)$. 
Moreover, we only need to know the formula for its restriction as 
$$ dd^c u |_{S_R} = d(d^c u)|_{S_R}.  $$
For the first term in equation (\ref{com-024}) we have 
\begin{equation}
\label{cal-001}
d\left\{  ( r u_r  ) d\eta  \right\}|_{S_R} = (r u_r)_{,\z}   d\z\wedge d\eta +  (r u_r)_{,\bar\z}   d\bar\z\wedge d\eta,
\end{equation}
where we used the notation
$$ (r u_r)_{,\z}  = \frac{ \d (r\d_r u)}{\d \z}, \ \ (r u_r)_{,\bar\z}  = \frac{ \d (r\d_r u)}{\d \bar\z}. $$
For the second term, we can write it as 
\begin{equation}
\label{cal-002}
d\left\{  ( r u_r  ) \cos\theta \frac{1}{2i}  \left(  \frac{d\bar\z}{\bar\z}   - \frac{d\z}{\z} \right)  \right\}|_{S_R}.
\end{equation}
Then we have 
\begin{equation}
\label{cal-003}
\d_{\z}\cos\theta =  \d_{\z} \left( \frac{1- |\z|^2}{ 1+ |\z|^2 } \right) = \frac{-2 \bar\z}{(1+ |\z|^2)^2},
\end{equation}
and the following is clear
\begin{equation}
\label{cal-004}
 d  \left( \frac{d\z}{\z} \right) =0, \ \ \ d  \left( \frac{d\bar\z}{\bar\z} \right) =0. 
\end{equation}
Thus equation (\ref{cal-002}) is equal to 
\begin{eqnarray}
\label{cal-005}
&& \frac{1}{2i} \left\{ (ru_r)_{,\z} \cos\theta \frac{\z}{|\z|^2}  -   (ru_r)  \frac{2}{(1+ |\z|^2)^2 }    \right\} d\z\wedge d\bar\z
\nonumber\\
&+&  \frac{1}{2i} \left\{ (ru_r)_{,\bar\z} \cos\theta \frac{\bar\z}{|\z|^2}  -   (ru_r)  \frac{2}{(1+ |\z|^2)^2 }    \right\} d\z\wedge d\bar\z
\nonumber\\
&=&  (ru_r )\frac{2id\z\wedge d\bar\z}{(1+|\z|^2)^2}  - \Re ( \z \cdot (ru_r)_{,\z}) \frac{\cos\theta}{|\z|^2} i d\z\wedge d\bar\z,
\end{eqnarray}
and then we obtain the following equality. 
\begin{lemma}
\label{lem-cal-001}
For any $u\in \cF^{\infty}(B_1)$, we have 
\begin{eqnarray}
\label{cal-006}
4 \ dd^c u|_{S_R} &=&  (r u_r)_{,\z}   d\z\wedge d\eta +  (r u_r)_{,\bar\z}   d\bar\z\wedge d\eta
\nonumber\\
& + & \left\{ 4 u_{,\z\bar\z} -  \Re ( \z \cdot (ru_r)_{,\z}) \frac{\cos\theta}{|\z|^2} +    \frac{2ru_r}{(1+|\z|^2)^2} \right\} i d\z \wedge d\bar\z. 
\end{eqnarray}
\end{lemma}
\begin{proof}
It is left to compute the third term in equation (\ref{com-024}), 
and we can perform like follows
$$ \frac{1}{2i}d (\d_{\z}u\cdot d\z - \d_{\bar\z}u\cdot d\bar\z )|_{S_R} = \frac{\d^2 u}{\d\z \d\bar\z }\  i d\z\wedge d\bar\z,  $$
and then the equality follows. 
\end{proof}
In terms of the real Hopf-coordinate, we observe that 
$$  i d\z \wedge d\bar\z = \frac{\sin\theta}{2 \cos^4(\theta/2)} d \theta \wedge d\vp, $$
and then we obtain another way to describe the above $2$-form as 
\begin{eqnarray}
\label{cal-007}
4 \ dd^c u|_{S_R} &=&  (r u_{, r\theta}) d\theta\wedge  d\eta + (ru_{,r\vp}) d\vp \wedge d\eta
\nonumber\\
& + & \left\{ (ru_r)\sin\theta - ( ru_{,r\theta} )\cos\theta \right\} d\theta \wedge d\vp
\nonumber\\
&+& 2\left\{ \sin\theta\cdot u_{,\theta\theta}  + \cos\theta \cdot u_{\theta} +  (\sin\theta)^{-1} u_{,\vp\vp}\right\} d\theta \wedge d\vp.
\end{eqnarray}

Next we are going to compute the following $3$-form. 
Combing Lemma (\ref{lem-com-001}) with Lemma (\ref{lem-cal-001}), it equals to 
$$ 16\ d^c u \wedge dd^c u|_{S_R} = \mbox{I} + \mbox{II},  $$ 
where 
\begin{eqnarray}
\label{cal-0075}
\mbox{I}:&=& \left( (r\d_r u) \left\{ d\eta -  \cos\theta \cdot \Im\left(  \frac{d\z}{\z} \right) \right\}  + 4 \Im (\d_{\z} u \cdot d\z ) \right)
\nonumber\\
&\wedge&   2 \Re \left\{ (r u_r)_{,\z}   d\z \wedge d\eta \right\} ,
\nonumber\\
\end{eqnarray}
and 
\begin{eqnarray}
\label{cal-0076}
\mbox{II}:&=& \left( (r\d_r u) \left\{ d\eta -  \cos\theta \cdot \Im\left(  \frac{d\z}{\z} \right) \right\}  + 4 \Im (\d_{\z} u \cdot d\z ) \right)
\nonumber\\
&\wedge &  \left\{ 4 u_{,\z\bar\z} -  \Re ( \z \cdot (ru_r)_{,\z}) \frac{\cos\theta}{|\z|^2} +    \frac{2ru_r}{(1+|\z|^2)^2} \right\} i d\z \wedge d\bar\z. 
\nonumber\\
\end{eqnarray}
Then the first term is 
\begin{eqnarray}
\label{cal-008}
\mbox{I}&=& d^c u \wedge  \left\{  (r u_r)_{,\z}   d\z\wedge d\eta +  (r u_r)_{,\bar\z}   d\bar\z\wedge d\eta \right\}
\nonumber\\
&=& \frac{i}{2}    (ru_r) \frac{\cos\theta}{|\z|^2}  \left\{    \bar\z\cdot (ru_r)_{,\bar\z}  +  \z\cdot (ru_r)_{,\z}   \right\} d\z \wedge d\bar\z\wedge d\eta
\nonumber\\
&-& 2i \left\{  ( \d_{\bar\z}u ) (r u_r)_{,\z}  +  (\d_{\z}u ) (r u_r)_{,\bar\z}   \right\} d\z \wedge d\bar\z\wedge d\eta
\nonumber\\
&=&  (ru_r) \frac{\cos\theta}{|\z|^2}  \Re \left\{   \z\cdot (ru_r)_{,\z}   \right\} i d\z \wedge d\bar\z\wedge d\eta
\nonumber\\
&-& 4 \Re\left\{    ( \d_{\bar\z}u ) (r u_r)_{,\z}  \right\}   i d\z \wedge d\bar\z\wedge d\eta,
\nonumber\\
 \end{eqnarray}
and the second term is 
\begin{eqnarray}
\label{cal-009}
\mbox{II}&=& d^c u \wedge 
\left\{ 4 u_{,\z\bar\z} -  \Re ( \z \cdot (ru_r)_{,\z}) \frac{\cos\theta}{|\z|^2} +    \frac{2ru_r}{(1+|\z|^2)^2} \right\} i d\z \wedge d\bar\z
\nonumber\\
&=& - (ru_r) \frac{\cos\theta}{|\z|^2}  \Re \left\{   \z\cdot (ru_r)_{,\z}   \right\} i d\z \wedge d\bar\z\wedge d\eta
\nonumber\\
&+&  \left\{ 4 (ru_r) u_{,\z\bar\z}  + \frac{2(ru_r)^2}{(1+ |\z|^2)^2}  \right\} i d\z \wedge d\bar\z\wedge d\eta. 
\nonumber\\
\end{eqnarray}
Observe that the first line on the R.H.S. of equation (\ref{cal-008}) 
cancels with the first line on the R.H.S. of equation (\ref{cal-009}), 
and then we eventually obtain the following formula. 
\begin{prop}
\label{prop-cal-001}
For any $u\in \cF^{\infty}(B_1)$, we have the following $3$-form on the $3$-sphere $S_R$,
in terms of the complex Hopf-coordinate as 
\begin{eqnarray}
\label{cal-010}
&&8 \ d^c u \wedge dd^c u|_{S_R}
\nonumber\\
&=& 2 \left(   (ru_r) u_{,\z\bar\z} - \Re\left\{ ( \d_{\bar\z}u ) (r u_r)_{,\z}  \right\} \right) i d\z \wedge d\bar\z\wedge d\eta
\nonumber\\
&+& \frac{(ru_r)^2}{(1+ |\z|^2)^2}   i d\z \wedge d\bar\z\wedge d\eta. 
\end{eqnarray}
\end{prop}
Moreover, we note that $\z$ is actually a coordinate on $\bC_{\infty} \cong \bC\bP^1$,
and this leads us to consider the  Fubini-Study metric on this K\"ahler manifold as 
$$ \omega: =  \frac{i d\z\wedge d\bar\z}{ 2 (1+ |\z|^2 )^2},$$
and it has volume $\int_{\bC\bP^1} \omega = \pi$. 
Then we can rewrite the above $3$-form in the following coordinate-free way 
\begin{eqnarray}
\label{cal-011}
&& 8 \ d^c u \wedge dd^c u|_{S_R}
\nonumber\\
&=&   2 (ru_r \cdot \Delta_{\omega} u ) \omega \wedge d\eta 
\nonumber\\
&-&  \{ \langle \nabla u , \nabla (ru_r) \rangle_{\omega} 
+\langle \nabla (ru_r), \nabla u  \rangle_{\omega}  \} \omega \wedge d\eta
\nonumber\\
&+&  2 (ru_r)^2 \omega \wedge d\eta, 
\end{eqnarray}
where $\nabla$ is the complex gradient for $\z$, and  the inner product is taken on $\bC\bP^1$ as 
$$\langle \nabla v , \nabla  w \rangle_{\omega} =  \mbox{tr}_{\omega} ( \d_{\z} v \wedge \dbar_{\z} w ).$$

Next, we are going to use another change of variables as $t: = \log r \in (-\infty, 0)$, 
and then $u$ can be rewritten as 
\begin{equation}
\label{cal-0111}
u_t(\z): = \hat{u}(t, \z)=  u( e^t, \z, \bar\z).
\end{equation}
Along each complex line through the origin of $\bC^2$, 
we have 
\begin{equation}
\label{cal-0115}
r \d_r u = \d_t \hat{u}  =  \dot{u}_t, 
\end{equation}
where $r = e^t$.
Therefore, we obtain the following decomposition formula for the complex Monge-Amp\`ere mass. 

\begin{theorem}
\label{thm-cal-001}
For any $u\in \cF^{\infty}(B_1)$, we have
\begin{equation}
\label{cal-012}
\frac{1}{\pi} \int_{S_R} d^c u \wedge dd^c u = 2 \int_{\bC\bP^1} ( \dot{u}_t \Delta_{\omega} u_t ) \omega  + \int_{\bC\bP^1} ( \dot{u}_t )^2 \omega,
\end{equation}
 where $\dot{u}_t = \frac{d u_t}{dt}|_{t=T}$ for $e^T = R$. 
\end{theorem}
\begin{proof}
The idea is to integrate both sides of equation (\ref{cal-011}) on the $3$-sphere $S_R$. 
Due to equation (\ref{hp-0016}) and Remark (\ref{rem-hp-001}),
it is legal to perform the integration
under the complex Hopf-coordinate $(\z, \eta) \in \bC_{\infty}\times S^1$.
After applying Fubini's Theorem, it boils down to 
take the following integration by parts on $\bC\bP^1$.

\begin{eqnarray}
\label{cal-013}
&-& \int_{\bC\bP^1} \d_{\z} u \wedge \dbar_{\z} (r u_r) -  \int_{\bC\bP^1} \d_{\z} (r u_r) \wedge \dbar_{\z} u
\nonumber\\
&=& \int_{\bC\bP^1} \d \dbar_{\z} u \cdot  (r u_r)  + \int_{\bC\bP^1} (r u_r) \cdot \d \dbar_{\z} u
\nonumber\\
&=& 2 \int_{\bC\bP^1}   ( r u_r \cdot  \Delta_{\omega} u )  \omega, 
\end{eqnarray}
where Stoke's theorem is used in the first equality. 
Equipped with the other two terms in the R.H.S. of equation (\ref{cal-011}),
and then our result follows. 

\end{proof}


In order to illustrate the above computation in a clear way,  
we will also invoke the real Hopf-coordinate, 
and perform the integration under it. 
First we note again that the function $u$ is periodic in the angle $\vp$ direction with period $2\pi$, 
and this can be seen directly as follows. 
\begin{eqnarray}
\label{cal-0135}
&& u \left(  |z_1|e^{\frac{i }{2}(\eta + \vp) + i\pi },  |z_2|e^{\frac{i}{2} (\eta - \vp) - i\pi} \right)
\nonumber\\
&=& u \left(  |z_1|e^{\frac{i }{2}(\eta + \vp) + i\pi },  |z_2|e^{\frac{i}{2} (\eta - \vp) + i\pi} \right)
\nonumber\\
&=& u \left(  |z_1|e^{\frac{i }{2}(\eta + \vp)  },  |z_2|e^{\frac{i}{2} (\eta - \vp) } \right). 
\end{eqnarray}

Then in the real Hopf-coordinate, 
we take the wedge product of equation (\ref{com-024}) and (\ref{cal-007}),
and obtain 
\begin{eqnarray}
\label{cal-014} 
&& 8\ d^c u \wedge dd^c u|_{S_R} 
\nonumber\\
&=& r u_r^2 \left\{ \frac{\d_{\vp}(u_{\vp} \cdot u^{-1}_r) }{\sin\theta}  
+ \sin\theta\cdot \d_{\theta} (u_{\theta} \cdot u^{-1}_r)   \right\} d\theta \wedge d \vp \wedge d\eta
\nonumber\\
&+& r u_r^2 \left\{       \cos\theta\cdot    (u_{\theta} \cdot u^{-1}_r)  + \frac{r}{2} \sin\theta    \right\} d\theta \wedge d \vp \wedge d\eta.
\end{eqnarray}

Next we will integrate the two sides of this equation on the $3$-sphere $S_R$. 
Thanks to Fubini's theorem, we can perform the integration by parts in the $\vp$-direction as follows


\begin{eqnarray}
\label{cal-016}
&& \int_{0}^{\pi} d\theta \int_0^{2\pi}  \sin^{-1}\theta(r u_r^2 ) \d_{\vp}(u_{\vp} \cdot u^{-1}_r)   d\vp
\nonumber\\
&=&  \int_{0}^{\pi} d\theta \int_0^{2\pi}   \sin^{-1} \theta \left\{ ( ru_r ) u_{,\vp\vp}  - r u_{\vp} u_{,r\vp}   \right\} d\vp
\nonumber\\
&=& 2  \int_{0}^{2\pi}  \int_0^{\pi}   ( ru_r )   \frac{u_{,\vp\vp}} {\sin\theta} d\theta d\vp. 
\end{eqnarray}
On the other hand, we can also perform the integration by parts in $\theta$-direction as follows.

\begin{eqnarray}
\label{cal-015}
&& \int_{0}^{2\pi} d\vp \int_0^{\pi}  \cos\theta(r u_r^2 )(u_{\theta} \cdot u^{-1}_r)   d\theta
\nonumber\\
&=& \int_{0}^{2\pi} d\vp  \left( \sin\theta (u_{\theta} \cdot r u_r)  \right) \big{\vert}_{0}^{\pi} 
- \int_{0}^{2\pi} d\vp \int_0^{\pi}  \sin\theta \d_{\theta}(u_{\theta} \cdot r u_r)   d\theta
\nonumber\\
&=& - \int_{0}^{2\pi} d\vp \int_0^{\pi}   \sin \theta \left\{  r \d_r (u_{\theta})^2  + r u_r^2  \cdot \d_{\theta}(u_{\theta} \cdot u^{-1}_r)   \right\} d\theta,
\nonumber\\
\end{eqnarray}
and the first term on the R.H.S. of equation (\ref{cal-015}) reads as 
\begin{eqnarray}
\label{cal-0155}
&& - \int_{0}^{2\pi} d\vp \int_0^{\pi}   \sin \theta \left\{  2r ( u_{\theta}  u_{,r\theta} )  \right\} d\theta
\nonumber\\
&=&  - \int_{0}^{2\pi} d\vp  \left(  2\sin\theta (r u_r \cdot u_{\theta}) \right) \big\vert_0^{\pi}
+ 2\int_{0}^{2\pi} \int_0^{\pi} (ru_r) \d_{\theta}(\sin\theta u_{\theta}) d\theta d\vp. 
\nonumber\\
\end{eqnarray}


Finally we combine equation (\ref{cal-014}), (\ref{cal-015}), (\ref{cal-0155}) and (\ref{cal-016}) together to obtain
\begin{eqnarray}
\label{cal-017}
&&  8 \int_{S_R} d^c u \wedge dd^cu 
\nonumber\\
&=& 8\pi  \int_0^{2\pi} \int_{0}^{\pi} (ru_r) \left\{ \ \frac{ u_{,\vp\vp}}{\sin\theta}    +    \d_{\theta}(\sin\theta u_{\theta})  \right\} d\theta  d\vp 
\nonumber\\
&+& 2\pi \int_0^{2\pi} \int_{0}^{\pi}  (r u_r)^2 \sin\theta d\theta d\vp
\nonumber\\
&=&  8\pi \int_{S^2}   (r u_r \cdot \Delta_{\Theta} u )  d\sigma_2  + 2 \pi \int_{S^2}  (r u_r)^2 d\sigma_2,
\end{eqnarray} 
where $d\sigma_2 = \sin\theta d\theta \wedge d\vp$ is the area form of the unit $2$-sphere $S^2$, 
and $\Delta_{\Theta}$ is the standard Laplacian on $S^2$ w.r.t. the round metric, i.e. we have
\begin{eqnarray}
\label{cal-018}
 \Delta_{\Theta} u
= \left\{ \frac{1}{ \sin\theta} \frac{\d}{\d\theta} \left( \sin\theta \frac{\d u}{\d\theta}   \right)    + \frac{1}{\sin^2\theta} \frac{\d^2 u }{\d \vp^2} \right\}. 
\end{eqnarray} 
Then it is clear that equation (\ref{cal-017}) is equivalent to (\ref{cal-012}) after the change of variables 
of the real and complex Hopf-coordinates and $r = e^t$.

\section{The residual Monge-Amp\`ere mass}
For the next step, we are going to estimate the two integrals 
on the R.H.S. of equation (\ref{cal-012}) as $t\rightarrow - \infty$. 
First, we note that the last integral has a deep connection with the Lelong number of $u$.

\subsection{The $L^2$-Lelong number}
\label{sub}
On the one hand, 
it is well known that the following limit is exactly the Lelong number at the origin of a plurisubharmonic function $u$ in $\bC^2$.
$$\nu_u(0) =  \lim_{r\rightarrow 0^+} \nu_u(0, r), $$
if we take 
\begin{equation}
\label{cvx-001}
\nu_u(0, r): = r\d^-_r \left(  \frac{1}{2\pi^2}  \int_{|\xi|= 1} u(r \xi)  \ d\sigma_3(\xi ) \right), 
\end{equation}
where $d\sigma_3$ is the area form of the unit $3$-sphere $S^3$. 
In fact, the integral on the R.H.S. of equation (\ref{cvx-001}) 
is $\log$-convex and non-decreasing in $r$. 
Therefore, the number $\nu_u(0, r)$ is non-negative and non-decreasing in the radial direction.

In fact, it is a standard fact that the Lelong number of a plurisubharmonic function $u$ is invariant under restriction to almost all complex
directions in $\bC^2$. 
From now on, we assume that the plurisubharmonic function $u$
is in the space $\cF(B_1)$ with the normalization $\sup_{B_1}u =-1$. 

Fixing a $\z\in \bC_{\infty}$, the complex line through the origin of $\bC^2$ in the direction  $\z$ can be written as 
$\l_{\z}: = (\lambda \z, \lambda)$ for all $\lambda\in \bC$ and $\z \neq \infty$. 
If $\z = \infty$, then we have $\l_{\infty}: = (\lambda, 0)$. 
Then the following restriction  
$$ u|_{\l_{\z}}  = u (\lambda \z, \lambda) =  u ( |\lambda| \z, |\lambda| ) $$
is actually a  non-decreasing convex function of $\log|\lambda|$ (see \cite{BB}),
and hence is also a non-decreasing convex function of 
$$t = \log r = \log |\lambda| + \frac{1}{2}\log(1+ |\z|^2), $$
for all $t\in (-\infty, 0)$.  
In other words, the negative function $u_t (\z) = \hat{u}(t, \z)$ (defined in equation (\ref{cal-0111}))
is also non-decreasing and convex in $t$.
Therefore, we have for each $\z$ fixed 
\begin{equation}
\label{cvx-0035}
\dot{u}_t \geq 0, \ \  and  \ \ \ddot{u}_t \geq 0,
\end{equation}
for almost all $t\in (-\infty, 0)$. 
Moreover, we infer from equation (\ref{cal-0115}) and (\ref{cvx-002}) 
that the Lelong number at zero of $u|_{\l_{\z}}$ is equal to 
\begin{equation}
\label{cvx-003}
 \nu_{u|_{\l_{\z}}} (0) = \lim_{t\rightarrow -\infty} \dot{u}_t (\z). 
 \end{equation}
Then the invariance of the Lelong number under restrictions reads as follows. 

\begin{lemma}[Remark (2.38), \cite{GZ}]
\label{lem-cvx-001}
For a plurisubharmonic function $u$, 
its Lelong number at the origin $\nu_{u}(0)$ is equal to $\nu_{u|_{\l_{\z}}} (0)$
for almost everywhere $\z\in\bC_{\infty}$. 
Moreover, for such a $\z$,  it is the decreasing limit of $\dot{u}_t (\z)$ as  $t\rightarrow - \infty$. 

\end{lemma}

Suppose that $u$ is further in $\cF^{\infty}(B_1)$. 
Recall that the total area of a unit $3$-sphere is $2\pi^2$.
Then we can rewrite the area form in the real Hopf-coordinate as 
\begin{equation}
\label{cvx-0011}
8\ d\sigma_3 =  \sin\theta d\theta \wedge d \vp \wedge d\eta, 
\end{equation}
and then equation (\ref{cvx-001}) can be re-written as 

\begin{equation}
\label{cvx-002}
\nu_u(0, r)=  \frac{1}{4\pi}  \int_{S^2} ( r \d_r u ) \ d\sigma_2 = \frac{1}{\pi}  \int_{\bC\bP^1} \dot{u}_t \omega. 
\end{equation}

Then we can obtain an estimate on the last term of equation (\ref{cal-012}),
and it will reveal that this term actually behaves like a kind of \emph{$L^2$-Lelong number}. 

\begin{lemma}
\label{lem-cvx-002}
For any $u\in \cF^{\infty}(B_1)$, we have 
\begin{equation}
\label{cvx-004}
 [\nu_u(0)]^2 = \lim_{t\rightarrow -\infty} \frac{1}{\pi}\int_{\bC\bP^1} (\dot{u}_t)^2 \omega. 
 \end{equation}
\end{lemma}
\begin{proof}
For each fixed $\z$, we denote $v_{\z}$ by the limit of the non-decreasing sequence $\dot{u}_t (\z)$ as 
$$ v_{\z} =   \lim_{t\rightarrow -\infty} \dot{u}_t (\z) \geq 0.$$ 
Thanks to Lemma (\ref{lem-cvx-001}), it coincides with the Lelong number $\nu_u(0)$
for almost everywhere $\z \in \bC\bP^1$.
Take 
\begin{equation}
\label{cvx-005}
M_A: = \max_{\z\in\bC\bP^1} \dot{u}_{-A}(\z), 
\end{equation}
for some constant $A>0$. Then we have 
$$ \dot{u}_t (\z) \leq M_A, $$
for all $t < -A$ and $\z\in \bC\bP^1$,
and then our result follows from the dominated convergence theorem as 
\begin{equation}
\label{cvx-006}
 \lim_{t\rightarrow -\infty} \frac{1}{\pi}\int_{\bC\bP^1} (\dot{u}_t)^2 \omega =  \frac{1}{\pi}\int_{\bC\bP^1} v^2_{\z} \omega =  \nu^2_u(0).
\end{equation}

\end{proof}

\subsection{Laplacian estimates}
The next goal is to estimate the first term on the R.H.S. of equation (\ref{cal-012}).
Before moving on, we need to take a closer look at the Laplacian $\Delta_{\omega}$ as follows.  

First, it is a standard fact that 
the Laplacian $\Delta_e$ with respect to the Euclidean coordinate in $\bR^4$ 
has the following decomposition in the hyper-spherical coordinate
\begin{eqnarray}
\label{lap-001}
\Delta_e = \Delta_r + r^{-2}\Delta_{\Xi},
\end{eqnarray}
where $\Delta_r  = r^{-3} \d_r (r^3 \d_r \cdot )$ is the radial part, 
and $\Delta_{\Xi}$ is the standard Laplacian on the unit $3$-sphere.  
That is to say, this operator $\Delta_{\Xi}$ only consists of derivatives 
in the directions perpendicular to the radial direction. 
In the real Hopf-coordinate, it is standard to compute (cf. \cite{LUW})
\begin{eqnarray}
\label{lap-002}
\frac{1}{4} \ \Delta_{\Xi} 
 =   \frac{1}{ \sin\theta} \frac{\d}{\d\theta} \left( \sin\theta \frac{\d}{\d\theta}   \right)    
+ \frac{1}{\sin^2\theta} \left( \frac{\d^2  }{\d \vp^2} +  \frac{\d^2  }{\d \eta^2} \right)
+ \frac{2\cos\theta}{\sin^2\theta} \frac{\d^2}{\d\vp \d\eta}.
\nonumber\\
\end{eqnarray}


If the function $u$ is also $S^1$-invariant, then we can compare equation (\ref{lap-002}) with (\ref{cal-018}), 
and obtain 
\begin{eqnarray}
\label{lap-003}
\Delta_{\Xi} u =  4\ \Delta_{\Theta} u = 2\ \Delta_{\omega} u. 
\end{eqnarray}
Therefore, we have the following Laplacian decomposition formula
for any function $u\in \cF^{\infty}(B_1)$

\begin{eqnarray}
\label{lap-004}
\Delta_e u = \Delta_r u + 2 r^{-2} \Delta_{\omega} u. 
\end{eqnarray}

\begin{lemma}
\label{lem-lap-001}
For any $u\in \cF^{\infty}(B_1)$, there exists a positive constant $M_A$ such that we have 
\begin{eqnarray}
\label{lap-0045}
2 \int_{\bC\bP^1} \dot{u}_t ( \Delta_{\omega} u_t ) \omega \leq 
 M_A \int_{\bC\bP^1} ( \ddot{u}_t +  2 \dot{u}_t) \omega,
\end{eqnarray}
for all $t \leq -A$. 
\end{lemma}
\begin{proof}
Thanks to the Laplacian decomposition formula (equation (\ref{lap-004})), 
the first integral on the R.H.S. of equation (\ref{cal-012})
can be rewritten as 
\begin{equation}
\label{lap-005}
2 \int_{\bC\bP^1} \dot{u}_t ( \Delta_{\omega} u_t ) \omega 
=  \int_{\bC\bP^1} \dot{u}_t  (r^2 \Delta_e u) \omega  - \int_{\bC\bP^1} \dot{u}_t  (r^2 \Delta_r u) \omega.
\end{equation}
Moreover, the second term on the R.H.S. of equation (\ref{lap-005}) is 
\begin{eqnarray}
\label{lap-006}
\int_{\bC\bP^1} \dot{u}_t  (r^2 \Delta_r u) \omega &=& \int_{\bC\bP^1} (ru_r) \left\{ (r^2 u_{,rr} + r u_r  ) + 2 ru_r \right\} \omega
\nonumber\\
&=& \int_{\bC\bP^1} \dot{u}_t  ( \ddot{u}_t +  2 \dot{u}_t) \omega. 
\end{eqnarray}
Thanks to equation (\ref{cvx-0035}),   
the integral in equation (\ref{lap-006}) is actually non-negative, 
and we obtain 
\begin{equation}
\label{lap-007}
2 \int_{\bC\bP^1} \dot{u}_t ( \Delta_{\omega} u_t ) \omega \leq  \int_{\bC\bP^1} \dot{u}_t  (r^2 \Delta_e u) \omega. 
\end{equation}
Furthermore, we have $\Delta_e u \geq 0$ since $u$ is plurisubharmonic on $\bC^2$. 
Take $M_A$ as the maximum of $\dot{u}_{-A} $ on $\bC\bP^1$ (equation (\ref{cvx-005})), and it follows for all $t < -A$ 
\begin{eqnarray}
\label{lap-008}
2 \int_{\bC\bP^1} \dot{u}_t ( \Delta_{\omega} u_t ) \omega & \leq&  \int_{\bC\bP^1} \dot{u}_t  (r^2 \Delta_e u) \omega
\nonumber\\
&\leq & M_A  \int_{\bC\bP^1}   (r^2 \Delta_r u + 2 \Delta_{\omega} u) \omega
\nonumber\\
&= & M_A \int_{\bC\bP^1} ( \ddot{u}_t +  2 \dot{u}_t) \omega. 
\end{eqnarray}
Here we have used 
$$\int_{\bC\bP^1} ( \Delta_{\omega} u ) \omega = 0 $$
in the last equality in equation (\ref{lap-008}).
Then our result follows. 

\end{proof}

Now everything boils down to study the asymptotic behavior of 
the following integral 
\begin{equation}
\label{lap-0081}
 \int_{\bC\bP^1} \ddot{u}_t  \omega = \frac{d}{dt} \left ( \int_{\bC\bP^1} \dot{u}_t \omega \right),
\end{equation}
as $t\rightarrow -\infty$. 
To this purpose, we introduce the following two non-negative functionals  
\begin{equation}
\label{lap-0082}
 I_u (t): =  \int_{\bC\bP^1} \dot{u}_t \omega; \ \ \ J_u (t) : =  \int_{\bC\bP^1}  (\dot{u}_t)^2 \omega, 
\end{equation}
and the negative functional as a primitive of $I_u$
\begin{equation}
\label{lap-0083}
\cI (u_t) : =  \int_{\bC\bP^1} u_t \omega, 
\end{equation}
for all $t\in (-\infty, 0)$. 
Then the following observation is crucial. 

\begin{prop}
\label{prop-lap-001}
Suppose $u\in\cF^{\infty}(B_1)$ has the Lelong number $\nu_u(0) \geq 0$ at the origin. 
Then there exists a sequence 
$t_i \in(-\infty, 0)$ converging to $-\infty$ such that
$$ \lim_{i\rightarrow +\infty} \frac{d I_u}{dt}   (t_i) = 0. $$
\end{prop}

\begin{proof}
First we note that $\cI(u_t)$
is a non-decreasing convex function along $t\in (-\infty, 0)$.
As its first derivative, 
$I_u $ is a non-negative, non-decreasing function in $t$. 
Moreover, we have 
\begin{equation}
\label{lap-009}
I'_u (t): = \frac{d I_u}{dt} (t)= \int_{\bC\bP^1} \ddot{u}_t  \omega \geq 0. 
\end{equation}
Thanks to equation (\ref{cvx-002}), we have its limit as a decreasing sequence in $t$ 
\begin{equation}
\label{lap-010}
\lim_{t \rightarrow  -\infty} I_u (t)  = \pi \nu_u (0) \geq 0.
\end{equation}
Hence $I_u$ is a non-negative $C^1$-continuous function 
decreasing to $\pi\nu_u(0)$ as $t\rightarrow -\infty$. 
Then there exists a $T_0<0$ such that we have 
$$ I_u(t) \leq \pi\nu_u(0) + 1, $$
for all $t< T_0$. 
Now it is sufficient to prove  
\begin{equation}
\label{lap-011}
\liminf_{t\rightarrow -\infty} I'_u (t) =0.
\end{equation}

Suppose not, and then there exists an $\ep >0$ 
such that 
\begin{equation}
\label{lap-012}
\liminf_{t\rightarrow -\infty} I'_u (t)  > \ep,
\end{equation}
and then there exists a $T_1 < 0$ such that for all $t < T_1$ 
\begin{equation}
\label{lap-013}
 I'_u (t)  > \ep /2. 
\end{equation}
Take $T_2: = \min\{ T_0, T_1\} -1$, 
and then the graph of $I_u$ will be under the following straight line for all $t < T_2$
$$ y (x)  = \frac{\ep}{2} x + I_u(T_2) - \frac{\ep T_2}{2} \leq \frac{\ep}{2} x + \pi\nu_u(0) +1 - \frac{\ep T_2}{2}, $$
but this implies $I_u (t) < 0$ for all $t$ negative enough, which is a contradiction. 

\end{proof}

Then we are ready to prove the first main theorem.

\begin{theorem}
\label{thm-001}
For any $u\in \cF^{\infty}(B_1)$, 
its residual Monge-Amp\`ere mass $\tau_u (0) $ is zero, if its Lelong number $\nu_u(0)$ is zero at the origin. 
\end{theorem}
\begin{proof}
Thanks to Proposition (\ref{prop-rm-001}) and the decomposition formula (equation (\ref{cal-012})), 
we have 
\begin{equation}
\label{lap-014}
 \pi^{-1}  \mbox{MA}(u)(B_r) = 2 \int_{\bC\bP^1} ( \dot{u}_t \Delta_{\omega} u_t ) \omega  + \int_{\bC\bP^1} ( \dot{u}_t )^2 \omega,
 \end{equation}
for $r = e^t$. 
Thanks to Lemma (\ref{lap-001}), we can further estimate
\begin{equation}
\label{lap-015}
\pi^{-1}  \mbox{MA}(u)(B_r) \leq  M_A \int_{\bC\bP^1} ( \ddot{u}_t +  2 \dot{u}_t) \omega + \int_{\bC\bP^1} ( \dot{u}_t )^2 \omega,
\end{equation}
for a uniform constant $M_A \geq 0 $ and all $t \leq -A$. 
In fact, we can re-write the above estimate (equation (\ref{lap-015}) as 

\begin{equation}
\label{lap-0155}
 \pi^{-1} \mbox{MA}(u)(B_r) \leq  M_A \left\{  I'_u (t) + 2 I_u (t) \right\} + J_u (t),
\end{equation}
for $r = e^t \leq e^{-A}$. 
Due to Lemma (\ref{lem-cvx-002}) and equation (\ref{cvx-002}),
the zero Lelong number $\nu_u(0) = 0 $ implies 
$$ I_u (t) \rightarrow 0, \ \ \  \mbox{and} \ \ \   J_u(t) \rightarrow 0, $$
as $t\rightarrow -\infty$. 
Moreover,  we infer from Proposition (\ref{prop-lap-001}) that 
there exists a sequence $t_i = \log r_i$ converging to $-\infty$ satisfying

\begin{equation}
\label{lap-016}
\pi^{-1} \mbox{MA}(u)(B_{r_i}) \leq  M_A \left\{   I'_u (t_i) + 2 I_u (t_i) \right\}  + J_u (t_i),
\end{equation}
and the R.H.S. of equation (\ref{lap-016}) converges to zero as $r_i \rightarrow 0$. 
However, the Monge-Amp\`ere mass of $u$ on the $r$-ball $\mbox{MA}(u)(B_r)$
is non-decreasing in $r$. 
Then it follows 
\begin{equation}
\label{lap-017}
\tau_u(0) = \frac{1}{\pi^2} \lim_{r\rightarrow 0^+} \mbox{MA}(u)(B_r)  = 0.  
\end{equation}

\end{proof}

\subsection{Maximal directional Lelong numbers}
Recall that the residual Monge-Amp\`ere mass at the origin of a plurisubharmonic function $u$ is 
equal to the following decreasing limit 
$$ \t_u (0) : = \frac{1}{\pi^2} \lim_{r\rightarrow 0^+} \mbox{MA}(u)(B_r).  $$

As we have seen in the proof of the above Theorem (\ref{thm-001}),
the inequality in equation (\ref{lap-0155}) has played a major role.  
In fact, this inequality implies a stronger result on the estimate of the residual mass. 
First we will re-state it as follows. 

\begin{lemma}
\label{lem-dn-000}
For a function $u\in \cF^{\infty}(B_1)$, we have 
\begin{equation}
\label{dn-001}
 \pi^{-1} \emph{MA}(u)(B_r) \leq  M_A \left\{  I'_u (t) + 2 I_u (t) \right\} + J_u (t),
\end{equation}
for all $ t= \log r \in ( -\infty, -A]$. 
\end{lemma}
Here the uniform constant $M_A $ is defined as 
$$ M_A: =  = \sup_{\z\in \bC\bP^1} \dot{u}_{-A}(\z) \geq 0, $$
for any $A>0$. 
Furthermore, it can be generalized to
to the family $\cF(B_1)$. 

\begin{defn}
\label{defn-dn-001}
For a function $u\in \cF(B_1)$,
the maximal directional Lelong number of $u$ at a distance $ A \in (0, \infty)$ to the origin 
is defined as 
$$ M_{A}(u): = \sup_{\z\in \bC\bP^1} \d_t^+u_t (\z)|_{t = -A} \in [0, +\infty]. $$
\end{defn}
As we have discussed in Section (\ref{sub}), 
the restriction $u|_{\l_{\z}}$ is a non-decreasing convex 
function along $t = \log r$. Hence the right derivative 
\begin{equation}
\label{dn-0010}
 \d^+_{t} u_t(\z)|_{t =t_0} =\d^+_{t} \hat{u}( t, \z)|_{t =t_0}  = \frac{d}{dt}|_{t= t_0^+} (u|_{\l_{\z}}) 
 \end{equation}
is well defined everywhere on $B^*_1$. 
Moreover, it is also non-negative and non-decreasing in $t$. 
Then it follows 
that this number $M_A(u)$ is non-negative, 
and non-increasing in $A$.  
Therefore, we can take its decreasing limit as $A\rightarrow +\infty$. 

\begin{defn}
\label{defn-dn-002}
For a function $u\in \cF(B_1)$,
the maximal directional Lelong number of $u$ at the origin 
is defined as 
$$ \lambda_u(0): = \lim_{A\rightarrow +\infty} M_A(u).$$
\end{defn}

There is no a priori reason that this number is finite, 
but it is for any $u$ in the sub-collection $\cF^{\infty}(B_1)$.
Moreover, the constant $M_A(u)$ coincides with the uniform constant 
$M_A$ for any $A>0$ in equation (\ref{dn-001}) in this case, 
and then we can obtain the following estimate.

\begin{theorem}
\label{thm-dn-001}
For a function $u\in\cF^{\infty}(B_1)$, we have 
\begin{equation}
\label{dn-002}
\t_{u}(0) \leq 2\lambda_{u}(0) \cdot \nu_u(0) + [\nu_u(0)]^2.  
\end{equation}
\end{theorem}
\begin{proof}
It is enough to prove the following inequality for all $A$ large
\begin{equation}
\label{dn-0020}
\t_{u}(0) \leq 2 M_A(u)\cdot \nu_u(0) + [\nu_u(0)]^2.
\end{equation}
Thanks to  Lemma (\ref{lem-cvx-002}),
we have 
$$  \pi \nu_u(0) = \lim_{t\rightarrow -\infty} I_u(t);  \ \ \   \pi [\nu_u(0)]^2 = \lim_{t\rightarrow -\infty} J_u(t). $$

Moreover, we can extract a sequence $t_i \rightarrow -\infty$
such that $I'_u(t_i) \rightarrow 0$ 
due to Proposition (\ref{prop-lap-001}). 
Then the result follows 
by taking limits on both sides of equation (\ref{dn-001}), 
where we have used the fact 
that the Monge-Amp\`ere 
mass $\mbox{MA}(u)(B_r)$
is non-decreasing in $r$. 

\end{proof}

It is apparent that 
the estimate in Theorem (\ref{thm-dn-001})
implies the zero mass conjecture for a function $u\in \cF^{\infty}(B_1)$.


 \section{The general case}
For a general plurisubharmonic function $u$ on a domain $D$, 
it is a standard fact that $u$ is in the Sobolev space $W_{loc}^{1,p}(D)$
for any $1\leq p <2$.  
If $D$ is in $\bC^2$ and $u$ is in the Cegrell class, cf. \cite{Ceg04}, 
then we can even take $p =2$ 
due to a result by Blocki, cf. \cite{Blo04}. 
In particular, for a function $u\in\cF(B_1)$,
the weak derivative $\nabla u$ 
exists almost everywhere as an $L^2_{loc}$-function on $B_1$. 

\subsection{The radial bound }

Writing $t = \log r$, the first partial $t$-derivative of $u$ also 
exists as an $L^2_{loc}$ function on $B_1$, since we have 
\begin{equation}
\label{gc-001}
 \d_t u (z) =  r\d_r u (z) = r\langle \nabla u (z), \hat{r}_{\vec{z}} \rangle, 
 \end{equation}
where  $\vec{z}$ is the  point vector of $z$ in $\bR^4$, 
and $\hat{r}_{\vec{z}}$ is the unit normal vector in the direction of $\vec{z}$. 
According to our definitions in Section (\ref{sub}), it is also equal to 
the $t$-derivative of $u$ under the restriction to the complex line $\l_{\z}$ as 
\begin{equation}
\label{gc-002}
\d_t u (z) = \d_t u(e^t, \z, \bar\z) = \dot{u}_t(\z)= \frac{d}{dt} u|_{\l_{\z}}, 
\end{equation}
for almost all $t\in (-\infty, 0)$ and $\z\in \bC\bP^1$, where $(r, \eta, \z, \bar\z)$ is the complex Hopf-coordinate of $z\in B_1^*$.  

Next we claim that this partial derivative $\d_t u$ is not merely $L^2$,
but actually bounded on $B^*_R$
for any $R\in (0,1)$, 
and this will be implied by the following crucial observation.  

\begin{lemma}
\label{lem-gc-001}
For a function $u\in \cF(B_1)$, its
maximal directional Lelong number $M_A(u)$ at distance $A$ is finite for any $A>0$. 
\end{lemma}
\begin{proof}
For each boundary sphere $S_R = \d B_R$ with $R\in (0,1)$,
we claim that 
there exists a constant $C_R>0$
such that it satisfies 
\begin{equation}
\label{gc-003}
u|_{S_R} \geq -C_R.
\end{equation} 
This is because $u$ is an $L^{\infty}_{loc}$-function in $B_1^*$, 
and hence it is in the space $L^{\infty}(S_R)$ for each $R\in (0,1)$.  
Then our claim follows since $u$ is 
actually everywhere defined as a locally bounded plurisubharmonic function on $B^*_1$.

Suppose on the contrary, we have 
$M_A(u) = +\infty$ for some $A>0$. 
Then there exists a sequence of points $\z_j \in \bC\bP^1$
such that we have 
\begin{equation}
\label{gc-004}
 \d_t^+ u_t(\z_j)|_{t = -A}  = \frac{d}{dt}|_{t= (-A)^+} \left( u|_{\l_{\z_j}} \right) 
\end{equation} 
diverges to $+\infty$ as $j\rightarrow +\infty$. 
In particular, we can pick a point $\xi\in \bC\bP^1$ 
among this sequence 
such that we have 
\begin{equation}
\label{gc-005}
 \d_t^+ u_t(\xi)|_{t = -A}  > \frac{2C_R}{A},
\end{equation} 
for $R= e^{-A}$. 
However, as a convex function of $t$, 
the graph of $u_t(\xi)$ is above the following straight line
$$ y(x) =  \frac{2C_R}{A} (x + A) - C_R, $$
 for all $t\in [-A, -A/2]$, 
and then we have 
$$ u_{-A/2}(\xi)  \geq y(-A/2) \geq 0, $$
but this contradicts to the normalization
$\sup_{B_1}u \leq -1$. 
Therefore, our result follows.

\end{proof}

One simple fact of
the number $M_A(u)$ is that 
it is always positive if we have a singularity of $u$ at the origin.

\begin{lemma}
\label{lem-gc-0015}
Suppose a function $u\in\cF(B_1)$ is equal to $-\infty$ at the origin. 
Then we have $M_A(u) >0$ for all $A>0$. 
\end{lemma}
\begin{proof}
Suppose on the contrary, we have $M_A(u)=0$ for some $A>0$. 
Then for each $\z\in \bC\bP^1$ fixed, 
it follows from the convexity of $u_t$ that we have 
$  u_t (\z) = u_{-A}(\z) $ 
for all $t\leq -A$. 
However, we have the convergence 
$$ \limsup_{z\rightarrow 0} u(z) = -\infty, $$
as a plurisubharmonic function. 
This forces that 
there exists a sequence of points $\z_j \in \bC\bP^1$ 
such that $u_{-A}(\z_j) \rightarrow -\infty$ as $j\rightarrow +\infty$. 
Hence it contradicts to the fact that there is a uniform lower bound of $u$
on the boundary sphere $S_R$ with $R = e^{-A}$.

\end{proof}

As a direct consequence of Lemma (\ref{lem-gc-001}),
the maximal directional Lelong number $\lambda_u (0)$ of a function $u\in \cF(B_1)$ 
at the origin is always finite, 
since it is the decreasing limit of $M_A(u)$ as $A\rightarrow +\infty$. 

\begin{cor}
\label{cor-gc-001}
For a function $u\in \cF(B_1)$, we have 
$$ 0\leq  \lambda_u(0) < +\infty. $$
\end{cor}

Another upshot of Lemma (\ref{lem-gc-001}) is
that $\d_t u$ is in the space $L^{\infty}(B^*_R)$ for all $R\in (0,1)$, 
since we have 
$$ \dot{u}_t (\z) \leq \d_t^+ u_t (\z), $$
whenever the L.H.S. exists. 
Then we are ready to 
 introduce the previous functionals 
on the boundary sphere $S_R$ for almost all $R$. 
To this purpose,  
a few facts of slicing theory (\cite{Dem2}, \cite{Fe}, \cite{Siu74}) will be invoked.  

\subsection{Slicing theory}
The radius function $r: B_1^* \rightarrow (0,1)$ can be viewed as a proper 
smooth submersion of differentiable manifolds. 
In fact, if we identify $B_1^*$ with $S^3\times (0,1)$, 
then $r$ is locally the projection map 
$$   S^3\times (0,1)\rightarrow (0,1) \ \ \  (r, \eta, \z, \bar\z) \rightarrow r. $$
Here we have used the complex Hopf-coordinate on $B_1^*$. 

A function $U\in L^1_{loc}(B_1^*)$ can be viewed as a order zero locally flat current, 
and then we can define its slicing for $r' \in (0,1)$ as 
$$ U_{r'}: = U|_{S_{r'}}.$$
Thanks to Fubini's Theorem, 
this restriction of $U$ to the fibers exists for almost everywhere $r'\in (0,1)$,
and we have $U_{r'}\in L^1(S_{r'})$. 
Moreover, its regularization $U_{\ep}: = U * \rho_{\ep}$ (equation (\ref{rm-006}))
has the property that $U_{\ep, r'} \rightarrow U_{r'}$ in $L^1(S_{r'})$
for almost all $r' \in (0,1)$. 
Furthermore, we have the following basic slicing formula 
 \begin{equation}
\label{gc-006}
\int_{B_1^*} U \a \wedge r^*\b = \int_{r'\in (0,1)} \left( \int_{S_{r'}} U_{r'}(\z, \eta) \alpha|_{S_{r'}}   \right) \beta (r'), 
\end{equation}
for every smooth $3$-form $\alpha$ compactly supported on $B_1^*$, and smooth $1$-form $\b$ on $(0,1)$.


Take a function $u$ in $\cF(B_1)$, 
and then 
it is in the space $L^p_{loc}(B_1)$ for any $p\geq 1$.
Therefore,  its slicing 
$u|_{S_R} $ 
exists as an $L^p$-function on $S_R$ 
for almost everywhere $R\in (0,1)$. 
Moreover, its $t$-derivative 
$\d_t u$ is in $L^{\infty}(B^*_{1-\delta})$ for a small $\delta>0$. 
Hence its slicing 
\begin{equation}
\label{gc-007}
(\d_t u)|_{S_R} = (r \d_r u )|_{S_R} = R (\d_r u)|_{S_R}
\end{equation}
exists as an $L^{\infty}$-function on $S_R$ for almost all $R\in (0, 1-\delta)$. 
Thanks  to the circular symmetry of $u$, 
the slices can also be viewed as $L^p$($L^{\infty}$)-functions on $\bC\bP^1$ via the Hopf-fiberation. 
Then we can rewrite for almost all $R = e^T$ 
\begin{equation}
\label{gc-008}
u_T: = u|_{S_R};\ \ \   \dot{u}_T: = (\d_t u)|_{S_R}. 
\end{equation}
Therefore, it is legal to introduce the following functionals 
$$ \cI(u_t): = \int_{\bC\bP^1} u_t \omega;  \ \ \  I_u(t): = \int_{\bC\bP^1} \dot{u}_t \omega; \ \ \  J_u(t): = \int_{\bC\bP^1} (\dot{u}_t )^2 \omega, $$
for almost all $t\in (-\infty, -1)$.
In the following, we summarize a few basic properties of these functionals: 

\begin{enumerate}
\item[\textbf{(a)}] 
the functional $\cI(u_t)$ is a negative, 
non-decreasing convex function along $t\in (-\infty, -1)$; 

\item[\textbf{(b)}] 
the functionals $I_u$ and $J_u$ are both 
non-negative and non-decreasing $L^{\infty}$-functions along $t\in (-\infty, -1)$; 

\item[\textbf{(c)}] 
the two functionals converge as $t\rightarrow -\infty$, 
$$ I_u(t) \rightarrow \pi \nu_u(0); \ \ \ J_u(t) \rightarrow \pi [\nu_u(0)]^2. $$

\end{enumerate}

Property (a) follows from a basic fact of plurisubharmonic functions,
and part (b) follows since $\dot{u}_t(\z)$ 
is a non-negative and non-decreasing function in $t$
for any fixed $\z\in\bC\bP^1$. 
Finally, property (c) follows from Lemma (\ref{lem-cvx-001}), (\ref{lem-gc-001})
and the dominated convergence theorem, 
exactly as we have used in Lemma (\ref{lem-cvx-002}).

\subsection{Regularization}

Thanks to Corollary (\ref{cor-rm-001}), the standard regularization 
 $u_{\ep}: = u* \rho_{\ep}$ 
 builds a sequence of functions in $\cF^{\infty}(B_1)$ decreasing to $u$. 
 Therefore, 
 its first derivative in the $t$-direction
 \begin{equation}
 \label{gc-0081}
\dot{u}_{\ep, t}: = \d_t u_{\ep} = r\d_r (u * \rho_{\ep})
 \end{equation}
is a non-negative, 
 non-decreasing smooth function along $t\in (-\infty, 0)$.  
Moreover, we will obtain a uniform control on $\dot{u}_{t,\ep}$ as follows. 

It follows from Lemma (\ref{lem-gc-001}) that 
the radial derivative $\d_r u$ is an 
$L^{\infty}_{loc}$-function on $B^*_{1-\delta}$ for any small $\delta>0$.  
Then the standard convolution $ (\d_r u)_{\ep}: = ( \d_r u ) * \rho_{\ep} $
converges to it strongly in $L^p$ for any $p\geq 1$
on a compact subset $K \subset B^*_{1-\delta}$, 
and hence its slicing $(\d_r u)_{\ep}|_{S_R}$
converges to $(\d_r u)|_{S_R}$ strongly in $L^p$ on $S_R$ for almost all $R\in (0, 1-\delta)$. 

However, a subtle issue left to us 
is to compare 
the convolution  
$(\d_r u)_{\ep}$ with $\d_r (u * \rho_{\ep})$.
Then we need the following version of Friedrichs' lemma.  

\begin{lemma}
\label{lem-gc-002}
For a function $u\in \cF(B_1)$ and a point $z\in B^*_{1- 2\delta}$, we have 
\begin{equation}
\label{gc-009}
\left| r \d_r (u * \rho_{\ep}) (z) - r(\d_r u * \rho_{\ep})(z) \right| \leq 2\ep || \nabla u ||_{L^1(B_{1-\delta})},
\end{equation}
for all $\ep < \min\{ |z|, \delta \}$.
\end{lemma}
\begin{proof}
First we evaluate 
\begin{eqnarray}
\label{gc-010}
\d_r (u* \rho_{\ep})(z) &=&  \left\langle \   \nabla_{\vec{z}} (u * \rho_{\ep})(z), \hat{r}_{\vec{z}}   \ \right\rangle
\nonumber\\
&=& \left\langle \ep^{-4} \int_{|z-y|< \ep} \nabla_{\vec{z}} \rho\left(  \frac{z -y }{\ep}\right)  u (y) d\lambda(y), \hat{r}_{\vec{z}}   \right\rangle
\nonumber\\
&=&  \left\langle -\ep^{-4} \int_{|z-y|<\ep} \nabla_{\vec{y}} \rho\left(  \frac{z -y }{\ep}\right)  u (y) d\lambda(y), \hat{r}_{\vec{z}}   \right\rangle
\nonumber\\
&=& \left\langle \ep^{-4} \int_{|z-y|< \ep}  \rho\left(  \frac{z -y }{\ep}\right)   \nabla u (y) d\lambda(y), \hat{r}_{\vec{z}}   \right\rangle
\nonumber\\
&=&  \int_{|w|<1}   \left\langle  \nabla u (z - \ep w),  \hat{r}_{\vec{z}}    \right\rangle  \rho\left( w \right)  d\lambda(w ),
\end{eqnarray}
where we used the change of variables $ y: = z - \ep w$. 
On the other hand, we have 
\begin{eqnarray}
\label{gc-011}
 (\d_r u* \rho_{\ep})(z) 
&=&  \ep^{-4} \int_{|z-y|<\ep} \left\langle \nabla u(y), \hat{r}_{\vec{y}}  \right\rangle \rho\left( \frac{z-y}{\ep} \right) d\lambda(y)
\nonumber\\
&=&  \int_{|w|< 1}  \left\langle \nabla u ( z- \ep w), \hat{r}_{ (\vec{z}- \ep \vec{w} )}   \right\rangle \rho(w) d\lambda (w). 
\end{eqnarray}
Take the difference of equation (\ref{gc-010}) and (\ref{gc-011}), and we obtain 
\begin{eqnarray}
\label{gc-012}
&& | (\d_r u* \rho_{\ep})(z) - \d_r (u* \rho_{\ep})(z) |
\nonumber\\
&\leq &  \int_{|w|< 1}  \left|  \left\langle \nabla u ( z- \ep w), \hat{r}_{ (\vec{z}- \ep \vec{w} )} - \hat{r}_{\vec{z}}   \right\rangle  \right| \rho(w) d\lambda (w)
\nonumber\\
&\leq&  C(\ep, z) \int_{|w|< 1}  \left| \nabla u ( z- \ep w)   \right| \rho(w) d\lambda (w)
\nonumber\\
&\leq&  C(\ep, z) || \nabla u ||_{L^1(B_{1-\delta}) }.
\end{eqnarray}
Here we have used the constant 
$$C(\ep, z): = \sup_{|w|<1} | \hat{r}_{ (\vec{z}- \ep \vec{w} )} - \hat{r}_{\vec{z}}  |, $$
and it can be controlled as 
\begin{eqnarray}
\label{gc-013}
&& | \hat{r}_{ (\vec{z}- \ep \vec{w} )} - \hat{r}_{\vec{z}}  | = \left|  \frac{z}{|z|} - \frac{z-\ep w}{|z - \ep w|}\right| 
\nonumber\\
&\leq & \frac{\ep}{|z|}  +  \left|  (z - \ep w) \left( \frac{1}{|z|} - \frac{1}{|z- \ep w|} \right)    \right| 
\nonumber\\
&\leq&   \frac{\ep}{|z|}  +  \frac{ \left|  |z| - |z- \ep w|   \right| }{|z|}   \leq  \frac{2\ep}{|z|},
\end{eqnarray}
Then our result follows. 

\end{proof}

Since the gradient of a function $u\in \cF(B_1)$ is also in $L^2_{loc}(B_1)$, 
Lemma (\ref{lem-gc-002}) implies that 
we can write 
\begin{equation}
\label{gc-014}
 r(\d_r u)_{\ep} (z) =  \dot{u}_{\ep, t}(z) + O(\ep), 
 \end{equation}
 for all $z$
in a punctured smaller ball in $B_1$. 
Hence we have proved the convergence 
\begin{equation}
\label{gc-015}
\dot{u}_{\ep, t} \rightarrow \d_t u 
\end{equation}
strongly in $L^p$ for any $p\geq 1$ on a relatively compact subset $K\subset B^*_1$. 
Then we can infer the following convergence of the functionals. 

\begin{lemma}
\label{lem-gc-003}
For a function $u\in \cF(B_1)$, 
and any constants $B> A >2$, 
there exists a subsequence $u_{\ep_j}$
such that 
we have the convergence 
$$ I_{u_{\ep_j}}(t) \rightarrow I_u(t); \ \ \    J_{u_{\ep_j}}(t) \rightarrow J_u(t),$$
for almost all $t\in [-B, -A]$. 
\end{lemma}
\begin{proof}
Let $\{\rho_j \}_{j=1}^k$ be a partition of unity of $\bC\bP^1$, 
and $\chi(r)$ a smooth cut-off function on the unit interval 
that $\chi(r) =1$ for all $r\in (e^{-B}, e^{-A})$ 
and $\chi(r) = 0$ outside of the interval $(e^{-B-1}, e^{-A +1})$. 
Then we introduce the following compactly supported $3$-forms on $B_1^*$ as 
$$  \a_j: = \chi(r) \rho_j(\z) d\sigma_3, $$
where $d\sigma_3$ is the area form of the unit $3$-sphere. 
Take the $1$-form as 
$$\b: = r^{-1}dr, $$
and then 
the basic slicing formula (equation (\ref{gc-006})) implies the following equality for $r = e^t$
\begin{equation}
\label{gc-016}
\int_{B_1} | \dot{u}_{\ep, t} - \d_t u |^2 \a_j \wedge \b = \int_{0}^{1} \left( \int_{S_r} | \dot{u}_{\ep, t} - \dot{u}_t |^2 \rho_j (\z) d\sigma_3 \right) \chi(r) r^{-1}dr.
\end{equation}
Summing up with $j$, we obtain 
\begin{equation}
\label{gc-017}
\int_{B_1} | \dot{u}_{\ep, t} - \d_t u |^2 \chi(r) d\lambda_t=  2\pi \int_{-\infty}^{0} \left( \int_{\bC\bP^1} | \dot{u}_{\ep, t} - \dot{u}_t |^2 \omega \right) \tilde\chi(t) dt, 
\end{equation}
where $d\lambda_t: = r^{-1} dr d\sigma_3$ is a measure on $\bC^2 - \{ 0 \}$, 
and $\tilde{\chi}(t): = \chi(e^t)$ is a cut-off function on $(-\infty, 0)$ 
that supports on $(-B-1, -A+1)$. 

From equation (\ref{gc-015}),
we infer that the L.H.S. of equation (\ref{gc-017}) 
converges to zero as $\ep\rightarrow 0$. 
Then it follows that we have 
$$ \dot{u}_{\ep, t} \rightarrow \dot{u}_t $$
strongly in $L^2$ on the fiber $\bC\bP^1\times \{ t \}$
for almost all $t\in [-B, -A]$, 
possibly after passing to a subsequence,
and then our result follows. 

\end{proof}

\subsection{Apriori estimates}

The following idea is to utilize the regularization  
$u_{\ep}\in \cF^{\infty}(B_1)$ to approximate a function $u\in \cF(B_1)$, 
and then apply our decomposition formula (equation (\ref{dn-001}))
 to this sequence to derive a contradiction. 
To this purpose, we first need a uniform estimate on the 
maximal directional Lelong numbers of the regularization. 

 \begin{lemma}
 \label{lem-app-001}
 Fix any two constants $ B> A  > 1 $. 
 Then there exists a uniform constant $C>0$
 such that we have 
  \begin{equation}
 \label{app-001}
M_B (u_{\ep}) \leq 2M_A (u) + C\ep,
 \end{equation}
 for all $\ep < \ep_0$, where 
 $$  \ep_0: = \frac{1}{2}\min \left\{ (e^{-A} - e^{-B}), e^{-B} \right\}. $$ 
 \end{lemma}
 \begin{proof}
The maximal directional Lelong number of the regularization can be written as 
 $$ M_B(u_{\ep}) = \sup_{\z\in \bC\bP^1} (\d_t u_{\ep, t} )(\z)|_{t = -B}.$$
Thanks to Lemma (\ref{lem-gc-002}), it boils down to estimate 
the partial derivatives 
$r ( \d_r u)_{\ep}(z)$ on the boundary sphere $S_R$ with $R = e^{-B}$.

Hence we take a point $z\in S_R$, and compute for all $\ep < \ep_0$ 
 \begin{eqnarray}
\label{app-002}
r ( \d_r u)_{\ep}(z): &=& r( \d_r u \ast \rho_{\ep})(z) 
\nonumber\\
&=&  \int_{|w|\leq 1} \frac{|z|}{|z - \ep w|}  (r\d_r) u(z- \ep w) \rho(w) d\lambda(w)
\nonumber\\
&=&  \int_{|w|\leq 1} \frac{|z|}{|z - \ep w|}   \dot{u}_t(\z) \rho(w) d\lambda(w), 
\end{eqnarray}
where $t= \log|z-\ep w| < -A$ and 
$\z\in \bC\bP^1$ is the angle corresponding to the point $z-\ep w$. 
Then consider the restriction of $u$ to the complex line $\l_{\z}$ through the origin. 
By utilizing the monotonicity of $\dot{u}_t(\z)$, 
we have 
$$ 0\leq \dot{u}_t (\z) \leq M_A (u), $$
for all $\z\in \bC\bP^1$.
Therefore, we further have 
 \begin{eqnarray}
\label{app-003}
r ( \d_r u)_{\ep}(z)
&\leq &  M_A (u) \int_{|w|\leq 1} \frac{|z|}{|z - \ep w|}  \rho(w) d\lambda(w)
\nonumber\\
&\leq &  (1+ 2\ep e^{-B}) M_A(u) \leq 2 M_A(u),
\end{eqnarray}
due to our choice of small $\ep$. 
Finally, our result follows from Lemma (\ref{lem-gc-002}).


 \end{proof}

\begin{rem}
\label{rem-app-001}
If we choose a smaller $\ep_0$ in Lemma (\ref{lem-app-001}),
then it is possible to obtain a shaper estimate. 
For instance, take an arbitrary real number $\b\in (0,1)$ and 
\begin{equation}
\label{app-0031}
 \ep_0: =  \min\left\{ \frac{1}{2} (e^{-A} - e^{-B}), \frac{\beta e^{-B}}{1+\b} \right\}.
 \end{equation}
 Then we have the following estimate for all $\ep < \ep_0$
 \begin{equation}
\label{app-0032}
 M_B (u_{\ep}) \leq (1+\b)M_A(u)+ C\ep,
 \end{equation}
 with the same constant $C$. 
 
\end{rem}

We have seen a useful estimate (equation (\ref{dn-001})),
based on the decomposition formula
of the Monge-Amp\`ere mass.
For a function $u\in \cF^{\infty}(B_1)$
and a constant $A>0$ with $M_A(u)>0$, 
we can re-write it as follows. 
\begin{equation}
\label{app-004}
I'_u(t) \geq   \frac{1}{M_A (u)} \left\{ \pi^{-1} \mbox{MA}(u)(B_r) - J_u(t) \right\}    - 2 I_u(t),
\end{equation}
for all $t = \log r \in (-\infty, -A]$. 
This new version of the estimate will be substantial
 for the following argument on the zero mass conjecture in the general case. 

\begin{theorem}
\label{thm-002}
For any $u\in \cF(B_1)$, 
its residual Monge-Amp\`ere mass $\tau_u (0) $ is zero, if its Lelong number $\nu_u(0)$ is zero at the origin. 
\end{theorem}

As in Theorem (\ref{thm-001}), 
a contradiction argument will be applied. 
In this case, we can assume $u(0) = -\infty$,
and then Lemma (\ref{lem-gc-0015}) implies that we have 
$M_A(u) >0$ for all $A>0$. 
Therefore, there exists a small number $\delta>0$ such that 
we can make the assumption 
\begin{equation}
\label{app-005}
\pi^{-1} \mbox{MA}(u)(B_r) \geq 20(2M_2(u) +1)\delta,
\end{equation}
for all $r\in (0,1)$. 

The \textbf{first step} is to 
fix a large constant $A>4$ 
such that the following condition holds: 

\begin{enumerate}
\item[\textbf{(i)}]
for almost all $t\leq -A$, we have 
$ 0 \leq I_u( t) \leq \delta$ and $ 0\leq J_u(t) \leq \delta $. 
\end{enumerate}
This follows from the fact $\nu_u(0) = 0$ and the 
properties (b) and (c)
of the functionals. 
Moreover, this condition will be kept for any larger constant. 

The \textbf{second step} is to pick up a subsequence $u_{\ep}$
of the regularization and a small $\ep_1>0$
such that  the following estimates hold
for all $\ep < \ep_1$: 
\begin{enumerate}
\item[\textbf{(ii)}]
we have at $t = -A$
$$ 0\leq I_{u_{\ep}}(-A)  \leq 2\delta, \ \ \ 0\leq J_{u_{\ep}}(-A)  \leq 2\delta;$$

\item[\textbf{(iii)}]
we have at $R = e^{-2A}$
$$ \pi^{-1} \mbox{MA}(u_{\ep})(B_R) \geq 10(2M_2(u) +1)\delta;$$

\item[\textbf{(iv)}]
for any $B\in [A, 2A]$, we have 
$$ M_B (u_{\ep}) \leq 2 M_2(u)+1. $$
\end{enumerate}

Condition-(ii) follows from 
Lemma (\ref{lem-gc-003}) and condition-(i),
since we have the convergence 
$$ I_{u_{\ep}}(t) \rightarrow I_{u}(t),\ \ \   J_{u_{\ep}}(t) \rightarrow J_{u}(t),$$
for almost all $t\in[-2A, A]$, possibly after passing to a subsequence. 
Here the constant $A$ may be perturbed to a slightly larger one, 
but this will not affect our condition-(i).

Condition-(iii) is implied by equation (\ref{app-005})
and our previous estimate (equation (\ref{rm-004})). 
Finally condition-(iv) follows from Lemma (\ref{lem-app-001}), 
by taking $\ep_1 < \min\{\ep_0, \frac{1}{C}\}$, where 
$C$ is the uniform constant appearing in the R.H.S. of equation (\ref{app-001}), and 
$$ \ep_0: = \frac{1}{2} \min \left\{  e^{-2} - e^{-4}, e^{-2A} \right\}. $$

Equipped with the above conditions, we continue the argument as follows. 

\begin{proof}
[Proof of Theorem(\ref{thm-002})]
Suppose on the contrary 
that $\tau_u(0)$ is positive. 
Then we can pick up a real number $\delta >0$ small enough 
such that equation (\ref{app-005}) holds for all $r\in(0,1)$. 
Choose a large constant $A$, 
and take a subsequence $u_{\ep}$ of the regularization
with $\ep< \ep_1$ such that condition (i)-(iv) are all satisfied.

Fix an $\ep< \ep_1$,
we will focus on the $I_{u_{\ep}}$-functional
in the closed interval $[-2A, A]$. 
First we claim $M_B(u_{\ep}) >0$ for all $B\in [A, 2A]$ under our choices,
and the reason is as follows. 

It is deduced from conditions (ii)-(iii) that we have for any $t\in [-2A, -A]$
\begin{eqnarray}
\label{app-006}
&& \pi^{-1} \mbox{MA}(u_{\ep})(B_{r} ) - J_{u_{\ep}}(t) 
\nonumber\\
&\geq&  \pi^{-1} \mbox{MA}(u_{\ep})(B_R) - J_{u_{\ep}}(-A) 
\nonumber\\
&\geq& 8 (2M_2(u) +1) \delta >0,
 \end{eqnarray}
where $r= e^{t}$ and $R = e^{-2A}$. 
However, 
this contradicts to the inequality in equation (\ref{dn-001}), 
if we have $M_B(u)=0$ for any $B\in [ A, -t]$. 
Then our claim follows. 

Thanks to the estimate in equation (\ref{app-004}),
 the first derivative of the $I_{u_{\ep}}$-functional can be estimated from below 
 as 

 \begin{eqnarray}
\label{app-007}
I'_{u_{\ep}}(t) &\geq&   \frac{1}{M_A (u_{\ep})} \left\{ 8(2M_2(u)+1)\delta \right\}    - 2 I_{u_{\ep}}(-A)
\nonumber\\
&\geq& \frac{1}{2 M_2(u) +1} \left\{ 8(2M_2(u)+1)\delta \right\}    -  4\delta 
\nonumber\\
&\geq& 4\delta,
\end{eqnarray}
for all $t\in [-2A, -A]$.
Here the first line on the R.H.S. of equation (\ref{app-007}) follows from equation (\ref{app-006}),
and the second line follows from condition-(iv). 
However, it implies 
\begin{equation}
\label{app-008}
I_{u_{\ep}} (-A) \geq \int_{-2A}^{A} I'_{u_{\ep}} (t) dt \geq 16 \delta, 
\end{equation}
which contradicts to the assumption on the $I_{u_{\ep}}$-functional in condition-(ii).
Then our result follows.

\end{proof}

\subsection{Positive Lelong numbers}
\label{sec02}
We note that the zero Lelong number condition 
does not really used in the apriori estimates during the proof of 
the Theorem (\ref{thm-002}). 
In fact, we can obtain the same estimate (equation (\ref{dn-002}))
in Theorem (\ref{thm-dn-001}) for a function $u$ in $\cF(B_1)$.
In other words, 
its residual Monge-Amp\`ere mass $\tau_u(0)$ 
can be controlled by the maximal directional Lelong number $\lambda_u(0)$
and Lelong number $\nu_u(0)$ 
at the origin.

\begin{theorem}
\label{thm-pln-001}
For a function $u\in\cF(B_1)$, we have 
\begin{equation}
\label{app-009}
\t_{u}(0) \leq 2\lambda_{u}(0) \cdot \nu_u(0) + [\nu_u(0)]^2.  
\end{equation}
\end{theorem}
\begin{proof}
First we note that 
it is enough to prove the inequality 
\begin{equation}
\label{app-010}
\t_{u}(0) \leq 2(1+\b) M_{A_0}(u) \cdot \nu_u(0) + [\nu_u(0)]^2,
\end{equation}
for all $A_0> 2$ large and $\b\in (0,1)$ small.
Moreover, we can also assume $u(0)= - \infty$,
and then $M_A(u)$ is always positive for all $A>0$.

Suppose on the contrary that equation (\ref{app-010}) fails. 
Then there exists a small number $\delta >0$ such that we have 
\begin{eqnarray}
\label{app-011}
&&\pi^{-1} \mbox{MA}(u)(B_r) 
\nonumber\\
&\geq&  2(1+\b) M_{A_0}(u) \cdot \pi \nu_u(0) + \pi [\nu_u(0)]^2+ 20(2M_{A_0}(u) +1)\delta,
\end{eqnarray}
for all $r\in (0,1)$. 
Next,  
we can pick up a large constant $A > 2A_0$,
and withdraw a subsequence $u_{\ep}$ with $\ep< \ep_1$
such that the following conditions are satisfied:

\begin{enumerate}
\item[\textbf{(v)}]
for almost all $t\leq -A$, we have 
$$ \pi\nu_u(0) \leq I_u( t) \leq  \pi\nu_{u}(0)+ \delta, \ \ \  \pi[\nu_u(0)]^2 \leq J_u( t) \leq  \pi[\nu_{u}(0)]^2+ \delta; $$

\item[\textbf{(vi)}]
we have at $t = -A$ and $t= -2A$
$$ \pi\nu_u(0) - \delta \leq I_{u_{\ep}}(-2A) \leq I_{u_{\ep}}(-A)  \leq \pi\nu_u(0)+ 2\delta$$
$$  \pi[\nu_u(0)]^2- \delta \leq J_{u_{\ep}}(-A)  \leq  \pi[\nu_u(0)]^2+ 2\delta;$$

\item[\textbf{(vii)}]
we have at $R = e^{-2A}$
\begin{eqnarray}
&&\pi^{-1} \mbox{MA}(u_{\ep})(B_R) 
\nonumber\\
&\geq&  2(1+\b) M_{A_0}(u) \cdot \pi \nu_u(0) + \pi [\nu_u(0)]^2+ 10(2M_{A_0}(u) +1)\delta;
\end{eqnarray}

\item[\textbf{(viii)}]
for any $B\in [A, 2A]$, we have for a small number $\k>0$
$$ M_B (u_{\ep}) \leq (1+\b) M_{A_0}(u)+ \k. $$
\end{enumerate}

As before, 
condition (v)-(vii) are satisfied 
due to Lemma (\ref{lem-gc-003})
and properties of $I_u$ and $J_u$ functionals. 
The points $t= -A$ and $t= -2A$ may also be slightly perturbed 
to satisfy condition-(vi). 
Moreover, condition-(viii) 
follows from Lemma (\ref{lem-app-001}) 
and remark (\ref{rem-app-001}), 
since we can take $\ep_1<\ep_0$ where 
$$ \ep_0: = \min\left\{ \frac{1}{2}(e^{-A_0} - e^{-2A_0}), \ \frac{\b e^{-2A}}{1+\b},\ \k C^{-1}\right\}.  $$

Equipped with these conditions (v)-(viii), 
and then we have the following estimate for all $t\in [-2A, -A]$

\begin{eqnarray}
\label{app-012}
&& \pi^{-1} \mbox{MA}(u_{\ep})(B_r) - J_{u_{\ep}}(t)
\nonumber\\
 &\geq& \pi^{-1} \mbox{MA}(u_{\ep})(B_R) - J_{u_{\ep}}(-A)
 \nonumber\\
 &\geq& 2(1+\b) M_{A_0}(u) \cdot \pi \nu_u(0) + 8(2M_{A_0}(u) +1)\delta >0,
\end{eqnarray}
where $r=e^t$ and $R= e^{-2A}$. 
Thanks to the estimate in equation (\ref{dn-001}) again,
we can infer $M_A(u_{\ep})>0$. 
Then the lower bound of $I'_{u_{\ep}}(t)$ can be estimated for all $t\in [-2A, A]$ as 

\begin{eqnarray}
\label{app-013}
I'_{u_{\ep}}(t)&\geq&\frac{2(1+\b) M_{A_0}(u) \cdot \pi \nu_u(0) + 8(2M_{A_0}(u) +1)\delta }{M_A(u_{\ep})} - 2I_{u_{\ep}}(-A)
\nonumber\\
&\geq& \frac{2(1+\b) M_{A_0}(u) \cdot \pi \nu_u(0) + 8(2M_{A_0}(u) +1)\delta }{(1+\b)M_{A_0}(u) + \k} - 2I_{u_{\ep}}(-A)
\nonumber\\
&\geq& 4\delta -\frac{2\pi\nu_u(0) \k}{ (1+\b)M_{A_0}(u) + \k}.
\end{eqnarray}

Finally, we take the integral
\begin{equation}
\label{app-014}
I_{u_{\ep}}(-A) = I_{u_{\ep}}(-2A) + \int_{-2A}^{A} I'_{u_{\ep}}(t) dt \geq \pi\nu_u(0) + 15\delta +O(\k),
\end{equation}
but this contradicts to condition-(vi), 
since we have $M_{A_0}(u)>0$ 
and $\k$ can be taken arbitrarily small. 
Then our result follows. 

\end{proof}

It is also apparent that the inequality (equation (\ref{app-009}))
in the Theorem (\ref{thm-pln-001}) implies the Theorem (\ref{thm-002}).  
In other words, this estimate on the residual 
Monge-Amp\`ere mass is a more general result than
the zero mass conjecture.

 In order to illustrate this inequality in a better way, 
we provide the following typical examples.

\begin{example}
Suppose a function $u\in \cF(B_1)$ is further radially symmetric, cf. \cite{Li19}.
That is to say, we have $u(z) = u(|z|)$ for all $z\in B_1$. 
It is a standard fact that we can compute 
$$ \tau_{u}(0) = [\nu_u(0)]^2, $$
and then equation (\ref{app-009}) holds in a trivial way. 
\end{example}

\begin{example}
Consider Demailly's example as follows, cf. \cite{Dem93}.
For any $\ep >0$, we take 
$$ u(z_1, z_2): = \max\{ \ep\log|z_1|, \ep^{-1}\log|z_2| \}.$$
Then we have $\nu_u(0) = \ep$ and $\tau_u(0) =1$. 
However, the maximal directional Lelong 
number is determined on the complex line $\{ z_1 =0 \}$.
Then it follows $M_A(u) = \ep^{-1}$ for each $A>0$, 
and equation (\ref{app-009}) holds as 
$$ 1 = \tau_u(0) \leq 2\lambda_{u}(0) \cdot \nu_u(0) + [\nu_u(0)]^2= 2+ \ep^2.  $$
\end{example}

\begin{example}
Consider Chi Li's example as follows, cf. \cite{Chi21}.  
Take the holomorphic projection $p: \bC^2 -\{ 0 \} \rightarrow \bC\bP^1 $,
and $\omega_0$ the standard Fubini-Study metric.
Let $\vp$ be the $\omega_0$-plurisubharmonic function on $\bC\bP^1 $ 
such that 
its singular locus $\{ \vp = -\infty \}$
is a Cantor set $\mathcal{C}$, and $\vp$ is $\omega_0$-pluriharmonic 
outside $\mathcal{C}$. 
Then we can construct 
$$ u(z):= \max\{ \log|z|^2 + p^*\vp, 2\log|z|^2 \}. $$
Then compute and obtain 
$\tau_u(0) = 8$ and $\nu_u(0) = 2$. 
Here we note that 
the function $p^* \vp$ does not directly contribute to the 
variation in the $t$-direction on each complex line. 
Therefore, we have $M_A(u) = 4$ for all $A>0$ large, 
and then it follows 
$$ 8 = \tau_u(0) \leq 2\lambda_{u}(0) \cdot \nu_u(0) + [\nu_u(0)]^2= 20. $$


\end{example}

\begin{example}
The following example is constructed by Coman and Guedj, cf. \cite{CG09}. 
We note that it is no longer $S^1$-invariant, and then equation (\ref{app-009}) fails in this case:
Fix an integers $n>4$, and we write 
$$ u(z): =  \frac{1}{2n} \log \left( | z_2 - z_1^n |^2 + |z_2^n|^2  \right).$$
Then it follows $\nu_u(0) = \frac{1}{n}$ and $\tau_u(0) =1$. 
The maximal direction of the Lelong numbers is along the line $\{ z_2 =0 \}$, 
and then we have $\lambda_u(0) =1$. 
Hence it follows 
$$ 1 = \tau_u(0) >   2\lambda_{u}(0) \cdot \nu_u(0) + [\nu_u(0)]^2 = \frac{2}{n} + \frac{1}{n^2}. $$
For $S^1$-invariant plurisubharmonic functions, we can consider the following: 
$$u_1(z): = \frac{1}{2n} \log \left( | z_2 - z_1 |^2 + |z_2^n|^2  \right).$$
Then we have $\nu_{u_1}(0) = \frac{1}{n}$, $\lambda_{u_1}(0) =1$ and $\tau_{u_1}(0) = \frac{1}{n}$, 
and equation (\ref{app-009}) holds as 
$$ \frac{1}{n} = \tau_{u_1}(0) \leq   2\lambda_{u_1}(0) \cdot \nu_{u_1}(0) + [\nu_{u_1}(0)]^2 =  \frac{2}{n} + \frac{1}{n^2}.$$
Another modification is 
$$ u_2(z): = \frac{1}{2n} \log \left( | z_2^n - z_1^n |^2 + |z_2^n|^2  \right).  $$
Then we have $\nu_{u_2}(0) = 1$, $\lambda_{u_2}(0) =1$ and $\tau_{u_2}(0) = 1$,  
and equation (\ref{app-009}) holds as 
$$ 1 = \tau_{u_2}(0) \leq   2\lambda_{u_2}(0) \cdot \nu_{u_2}(0) + [\nu_{u_2}(0)]^2 = 3.$$

\end{example}

\section{Variational approach}
In this section, we will introduce another 
point of view to look at the decomposition formula 
(Theorem (\ref{thm-cal-001})) and also the zero mass conjecture. 
 The observation is that the function $u_t$ actually 
defines a curve in the 
space of all quasi-plurisubharmonic functions on $\bC\bP^1$, 
and it can be thought of as a subgeodesic with non-trivial $S^1$-fiberation 
 in the space of 
K\"ahler potentials. 
Then the first term on the R.H.S. of equation (\ref{cal-012})
corresponds to  the first variation of the so called \emph{pluri-complex energy}.

 \subsection{Subgeodesic on fiber bundles}
 We recall some basic facts in K\"ahler geometry. 
 Let $X$ be a compact Riemann surface without boundary, 
 and $\omega_0$ a K\"ahler metric on this Riemann surface.
 Denote $\mathbb{D}^*$ by the punctured unit disk in $\bC$,
 and it can be identified with the product $(0,1)\times S^1$ via the polar coordinate $(r, s)$,
 namely, we can write $$z:= r e^{i s} $$ as a complex variable in $\bD^*$. 
 
 Consider the product space $Y: = \bD^* \times X$, 
 and the projection maps $\pi_1: Y \rightarrow \bD^*$ and $\pi_2: Y \rightarrow X$. 
 The pull back $\pi_2^* \omega_0$ is a closed non-negative $(1,1)$-form on $Y$. 
 Then a \emph{subgeodesic ray} in the space of K\"ahler potentials  is 
 an $S^1$-invariant $\pi^*_2 \omega_0$-plurisubharmonic functions $v$ on $Y$.
 That is to say, 
 if we take any local potential $\Phi$ of $\omega_0$,
then the function 
 $$ V: = \pi_2^* \Phi + v $$ 
 is independent of  the variable $s$, 
 and  plurisubharmonic in the product manifold $(0,1)\times S^1 \times X$.
In other words, a subgeodesic $V$ is locally an $S^1$-invariant plurisubharmonic function 
on the trivial $\bD^*$-bundle of $X$,
and hence the restriction $v|_{X\times\{t\}}$ is a $\omega_0$-plurisubharmonic function 
on each fiber $X\times\{ t\}$.

 
 Furthermore, a subgeodesic ray $v$ is a \emph{geodesic ray}, 
 if it satisfies the following \emph{homogeneous complex Monge-Amp\`ere equation} on $Y$
 \begin{equation}
 \label{sub-001}
( dd^c_{z, X} V   )^2 = (\pi_2^* \omega_0 + dd^c_{z, X} v)^2 = 0. 
 \end{equation}

 From now on, we put $X = \bC\bP^1$ and $\omega_0$ the Fubini-Study metric on it. 
 In fact, the projective space can be viewed as the moduli space of $\bC^2-\{0\}$,
 under the natural $\bC^*$-action. 
The punctured disk $\bD^*$ acts in the same way on 
the punctured ball $B_1^*$ in $\bC^2$,
and then $B_1^*$ can be thought of as a non-trivial $\bD^*$-bundle of $\bC\bP^1$
via the Hopf-fiberation, i.e. we can write the bundle map as follows
$$(\dag) \ \ \ \ \ \ \ \ \  \bD^* \hookrightarrow B_1^* \xrightarrow{p} \bC\bP^1. $$



If we take a function $u\in \cF(B_1)$,
then it is naturally an $S^1$-invariant plurisubharmonic function on this non-trivial $\bD^*$-bundle. 
Writing $B_1^*$ as a product $(0,1)\times S^3$,
a fiber $\{r \}\times S^3$ is identified with the $3$-sphere $S_r $ for all $r\in (0,1)$. 
Then 
the restriction $u|_{S_r}$ 
can be viewed as a function on $\bC\bP^1$ via the Hopf-fiberation. 
In fact,  
this restriction $u|_{S_r}$ is exactly the function $u_t$ defined in equation (\ref{cal-0111}), 
via the change of variables $r = e^t$.    

If we further require $u\in \cF^{\infty}(B_1)$, 
then the Laplacian decomposition formula (equation (\ref{lap-004})) implies 
on each fiber $\bC\bP^1 \times \{ t\}$
\begin{equation}
\label{cf-001}
 \frac{1}{2}\left( \ddot{u}_t + 2\dot{u}_t \right) + \Delta_{\omega} u_t  \geq 0. 
\end{equation}
Therefore, the function $u_t$ is 
quasi-plurisubharmonic on each fiber $\bC\bP^1 \times \{t \}$,
but the lower bound of its complex hessian varies 
with respect to $t$ and $u$ itself. 
For these reasons, we introduce the following definition. 
\begin{defn}
\label{def-cf-001}
For a function $u\in\cF(B_1)$,
we say that $u_t$ is a bounded subgeodesic ray 
on the fiber bundle $\emph{(\dag)}$.
Moreover, it is a geodesic ray on this fiber bundle if we have 
$$ (dd^c u)^2 =0,$$
on $B_1^*$. 
\end{defn}

If $u$ is further in the family $\cF^{\infty}(B_1)$,
then we say that $u_t$ 
is a $C^2$-continuous subgeodesic ray 
on this fiber bundle.

\subsection{Energy functionals}
\label{sub-001}
For a quasi-plurisubharmonic function on $\bC\bP^1$,
the pluri-complex energy $\cE$ is defined as 
$$\cE(u_t) : = \int_{\bC\bP^1} (-u_t) dd^c_{\z} u_t =  - \int_{\bC\bP^1}  u_t (\Delta_{\omega}u_t ) \omega. $$

It is well known that this energy is concave along a subgeodesic in the space of K\"ahler potentials. 
In fact, the push-forward of the Monge-Amp\`{e}re measure 
is exactly the complex hessian of $-\cE$ along a sub-geodesic ray, cf. \cite{BB22}. 

Next we will show that our decomposition formula is an analogue of this
on a non-trivial fiber bundle. 
Let $u_t$ be a $C^2$-continuous subgeodesic ray on the fiber bundle $(\dag)$. 
Then the first variation of this energy with respect to $t$ can be computed as 
\begin{equation}
\label{sub-004}
\frac{d}{dt} \cE (u_t) = - 2 \int_{\bC\bP^1}  \dot{u}_t (\Delta_{\omega}u_t ) \omega,
\end{equation}
and this is exactly the negative of the first term on the R.H.S. of the decomposition formula (equation (\ref{cal-012})). 
Take the following non-negative functional to represent the complex Monge-Amp\`ere mass for $r = e^t$
\begin{equation}
\label{sub-0045}
K (u_t): = \pi^{-1} \cdot \mbox{MA}(u)(B_r).
\end{equation}
Then the decomposition formula can be rewritten as follows. 

\begin{lemma}
\label{lem-en-000}
Suppose $u_t$ is a $C^2$-continuous subgeodesic ray on the fiber bundle $(\dag)$. 
Then we have 
\begin{equation}
\label{sub-005}
- \frac{d}{dt }\cE (u_t) + J_u(t) = K (u_t),
\end{equation}
for all $t\in (-\infty, 0)$. 
\end{lemma}
By invoking toric plurisubharmonic functions,  
we can take a closer look at these energy functionals. 

\begin{example}
\label{exa-sub-001}
Suppose $u\in\cF^{\infty}(B_1)$ also has toric symmetry, i.e. we have 
$$ u(z_1, z_2) = u(e^{i\theta_1}z_1, e^{i\theta_2}z_2), $$
for arbitrary $\theta_1, \theta_2\in \bR$. 
Then in the real Hopf-coordinate, we can write 
$$ u(r, \eta, \theta, \vp) = u(r, \theta). $$
On the one hand, equation (\ref{cal-017}) boils down to the following
\begin{eqnarray}
\label{sub-006}
&&  8 \int_{S_R} d^c u \wedge dd^cu 
\nonumber\\
&=& 16\pi^2  \int_{0}^{\pi} (ru_r)  \d_{\theta}(\sin\theta u_{\theta}) d\theta
+ 4\pi^2  \int_{0}^{\pi}  (r u_r)^2 \sin\theta d\theta. 
\end{eqnarray} 
Then the first variation of the energy $\cE$ with respect to $t$ is 
$$ - \frac{d}{dt} \cE(u_t) =   2\pi \int_{0}^{\pi} (ru_r)  \d_{\theta}(\sin\theta u_{\theta}) d\theta, $$
and the $L^2$-Lelong number is 
$$ J_u(t) = \frac{\pi}{2} \int_{0}^{\pi}  (r u_r)^2 \sin\theta d\theta.  $$

On the other hand, 
we can also compute the complete complex Hessian of $u$ as follows
\begin{eqnarray}
\label{sub-007}
  4\ dd^cu 
&=& \left\{   2\sin\theta u_{,r\theta} - \d_r(ru_r) \cos\theta \right\} dr \wedge d\vp + \d_r(ru_r) dr\wedge d\eta
\nonumber\\
&+&  \left\{ (ru_r)\sin\theta - ( ru_{,r\theta} )\cos\theta \right\} d\theta \wedge d\vp
\nonumber\\
&+& 2\left\{ \sin\theta\cdot u_{,\theta\theta}  + \cos\theta \cdot u_{\theta} \right\} d\theta \wedge d\vp + (ru_{,r\theta}) d\theta\wedge d\eta. 
\end{eqnarray} 
Therefore, we obtain a formula for the positive measure as 
\begin{eqnarray}
\label{sub-008}
&&8 (dd^c u)^2
\nonumber\\
&=& \left\{ 2\d_r(ru_r) \d_{\theta}(\sin\theta u_{\theta}) -2 r\sin\theta(u_{,r\theta})^2 + \d_r(ru_r) (ru_r)\sin\theta  \right\} 
\nonumber\\
&&  dr\wedge d\theta \wedge d\vp \wedge d\eta
\end{eqnarray}
Then we continue to compute the second variation of the energy $\cE$ as 
\begin{eqnarray}
\label{sub-009}
- \frac{d^2 \cE(u_t)}{ dt^2}
&=& 2\pi r \int_0^{\pi} \left\{  \d_r (ru_r) \d_{\theta}(\sin\theta u_{\theta}) + (ru_r) \d_{\theta} (\sin\theta u_{,r\theta}) \right\} d\theta
\nonumber\\
&= & 2\pi r \int_0^{\pi} \left\{  \d_r (ru_r) \d_{\theta}(\sin\theta u_{\theta})  -  r \sin\theta (u_{,r\theta})^2 \right\} d\theta
\nonumber\\
&\geq& - \pi r \int_0^{\pi}   \dot{u}_t \ddot{u}_t \sin\theta d\theta. 
\end{eqnarray}

\end{example}

As we have expected, the last term on the R.H.S. of equation (\ref{sub-009})
is exactly equal to $- d J_u(t) / dt$. 
That is to say, the almost concavity of the $\cE$ functional along the ray $u_t$
is due to the positivity of $(dd^c u)^2$.

In fact, the functionals $J_u(t)$ and $K(u_t)$ 
are both non-negative, non-decreasing functions along $t\in (-\infty, 0)$. 
Then  denote $\mathcal{K}$, $\mathcal{J}$ by 
their primitives along $t$, i.e. we have 
$$ \frac{d}{dt} \mathcal{K}(u_t) =  K (u_t), \ \ \mbox{and } \ \ \ \frac{d}{dt} \mathcal{J}(u_t) =  J_u (t),$$
for all $t\in (-\infty, 0)$. 
Then both functionals are non-decreasing, convex functions along $t$.
Therefore, we can rephrase Lemma (\ref{lem-en-000}) as follows. 

\begin{theorem}
\label{thm-en-001}
Suppose $u_t$ is a $C^2$-continuous subgeodesic ray on the fiber bundle $(\dag)$.
Then we have the equality 
\begin{equation}
\label{en-002}
\mathcal{K}(u_t) = \mathcal{J}(u_t) - \mathcal{E}(u_t) + A,
\end{equation}
for a constant $A$ and all $t\in (-\infty, 0)$. 
In particular, 
the difference $\mathcal{J} - \cE $ is non-decreasing and convex
along the ray.   
Moreover, it is affine if and only if $u_t$ is a geodesic ray on this fiber bundle. 

\end{theorem}

\begin{proof}
It is left to prove the statement about the geodesic. 
For a function $u\in \cF^{\infty}(B_1)$, it is clear that 
$ K(u_t) = 0,$
for all $t\in (-\infty, 0)$ if and only if $(dd^c u)^2 =0$ on $B_1^*$. 
Then our result follows. 
\end{proof}

Recall that a primitive of the $I_u$-functional is defined as 
$$ \mathcal{I}(u_t) = \int_{\bC\bP^1} u_t \omega. $$
It is also non-decreasing and convex along $t$, and 
the asymptote of $\pi^{-1}\cdot \cI$ as $t\rightarrow -\infty$ is exactly the Lelong number of $u$ at the origin. 
Then the zero mass conjecture
can be rephrased in the energy setting, 
and it reveals the relation between the asymptotes of these two 
functionals along a subgeodesic ray. 

\begin{theorem}
\label{thm-en-002}
Suppose $u_t$ is a bounded subgeodesic ray on the fiber bundle $(\dag)$.
Assume that we have 
$$ \lim_{t\rightarrow -\infty}\frac{d\cI(u_t)}{dt} =0. $$
Then it follows 
$$ \lim_{t\rightarrow -\infty}\frac{d\cK(u_t)}{dt} =0. $$


\end{theorem}

Similarly, 
for a general function $u\in \cF(B_1)$,
the estimate on the residual Monge-Amp\`ere mass 
 in Theorem (\ref{thm-pln-001}) can also be interpreted 
under the energy picture. 
 It says that 
  the asymptote of the $\mathcal{K}$-functional 
 can be controlled by the maximal directional 
 Lelong number and the asymptote of the $\cI$-functional along the ray.

\begin{rem}
\label{rem-000}
It is possible to prove equation (\ref{en-002}) along a 
bounded subgeodesic ray $u_t$.
In fact, the pluri-complex energy $\cE$ can be rewritten as 
$$ \cE(u_t) = \int_{\bC\bP^1} d_{\z} u_t \wedge d^c_{\z} u_t. $$
This is well defined for almost all $t$
since $u$ is in the Sobolev space $W^{1,2}_{loc}(B_1)$, cf. \cite{Blo04}. 
Then one can utilize the regularization sequence $u_{\ep}$ 
to obtain the convergences on both sides of equation (\ref{en-002}) as before.  
\end{rem}

\end{document}